\documentclass[11pt,           
english                        
]{article}

\usepackage[english,german,strings]{babel}     
\usepackage[utf8]{inputenc}    
\usepackage[T1]{fontenc}       
\usepackage{longtable}         
\usepackage{exscale}           
\usepackage[final]{graphicx}   
\usepackage[sort]{cite}        
\usepackage{array}             
\usepackage{fancyhdr}          
\usepackage[a4paper]{geometry} 
\usepackage[multiuser]{fixme}  
\usepackage{xspace}            
\usepackage{tikz}              
\usepackage{tikzsymbols}       
\usepackage{nchairx}           
\usepackage{wasysym}           
\usepackage[sectionbib         
           ]{chapterbib}       
\usepackage{appendix}
\usepackage[expansion=false    
           ]{microtype}        
\usepackage[nottoc]{tocbibind} 
\usepackage[backref=page,      
           final=true,         
           pdfpagelabels       
           ]{hyperref}         

\usepackage{tikz-cd}       
\usepackage{xcolor} 
\usetikzlibrary{arrows,calc,through,backgrounds,matrix,
	decorations.pathmorphing,positioning,babel}    

\newcommand{\Mc}{\mathsf{MC}}   
\newcommand{\Def}{\mathsf{Def}} 
\newcommand{\poly}{{\scriptscriptstyle{\mathrm{poly}}}}    
\newcommand{\Tpoly}{T_\poly}  
\newcommand{\Dpoly}{D_\poly}  
\newcommand{\Tpolyres}{\mathcal{T}_\poly}  
\newcommand{\Dpolyres}{\mathcal{D}_\poly}  

\title{The Homotopy Class of twisted $L_\infty$-morphisms }
\author{
  \textbf{Andreas Kraft}\thanks{\texttt{akraft@unisa.it}},\\[0.3cm]
   Dipartimento di Matematica\\
   Università degli Studi di Salerno\\
   via Giovanni Paolo II, 123\\
   84084 Fisciano (SA)\\
   Italy \\[0.5cm]
  \textbf{Jonas Schnitzer}\thanks{\texttt{jonas.schnitzer@math.uni-freiburg.de}},\\[0.3cm]
  Department of Mathematics\\
   University of Freiburg\\
   Ernst-Zermelo-Straße, 1\\
	D-79104 Freiburg\\
   Germany
}

\begin{document}
\selectlanguage{english}

\maketitle

\abstract{
The global formality of Dolgushev depends on the choice of a 
torsion-free covariant derivative. 
We prove that the globalized formalities with respect to two different covariant 
derivatives are homotopic. More explicitly, we derive the statement by proving a more  
general homotopy equivalence between $L_\infty$-morphisms that are twisted with 
gauge equivalent Maurer-Cartan elements.
}

\tableofcontents

%
%
\section{Introduction}

The celebrated formality theorem by Kontsevich \cite{kontsevich:2003a} provides the 
existence of an $L_\infty$-quasi-isomorphism from the differential graded 
Lie algebra (DGLA) of polyvector fields $\Tpoly(\mathbb{R}^d)$ to the DGLA of polydifferential operators $\Dpoly(\mathbb{R}^d)$. In 
\cite{dolgushev:2005a,dolgushev:2005b} Dolgushev globalized this result to 
general smooth manifolds $M$ using a geometric approach. Being a quasi-isomorphism, this formality induces a bijective correspondence
\begin{equation}
	\label{eq:IntroFormality}
	\boldsymbol{U}\colon
	\Def(\Tpoly(M)[[\hbar]]) 
	\longrightarrow
	\Def(\Dpoly(M)[[\hbar]])
\end{equation} 
between equivalence classes $\Def(\Tpoly(M)[[\hbar]])$ of formal Poisson structures 
$\hbar\pi \in \Secinfty(\Anti^2 TM)[[\hbar]]$ on $M$ and equivalence classes 
$\Def(\Dpoly(M)[[\hbar]])$ of star products $\star$ on $M$, see also 
\cite{canonaco:1999a,manetti:2005a} for more details on deformation theory. 
In particular, 
this associates to a classical Poisson structure $\pi_\cl$ a class of deformation 
quantizations $\boldsymbol{U}([\hbar\pi_\cl])$ in the sense of the seminal paper 
\cite{bayen.et.al:1978a}. On 
the other hand, it also gives a way to assign to each star product a class 
of formal Poisson structures, the so-called \emph{Kontsevich class} of the 
star product.

However, the above mentioned globalization procedure of the 
Kontsevich formality from $\mathbb{R}^d$ to a general manifold $M$ 
discussed in \cite{dolgushev:2005a}  depends on the choice of a 
torsion-free covariant derivative. More explicitly, it uses the covariant 
derivative to obtain Fedosov resolutions of the polyvector 
fields and polydifferential operators between which one has a fiberwise Kontsevich 
formality. Recently, in 
\cite[Theorem~2.6]{bursztyn.dolgushev.waldmann:2012a} 
it has been shown that the map $\boldsymbol{U}$ 
from \eqref{eq:IntroFormality} does not depend on the choice of the connection.  
In this paper we investigate the role of the covariant 
derivative at the level of the formality and not
at the level of equivalence classes of Maurer-Cartan elements. 

The key point is that changing the covariant 
derivative corresponds to twisting by a Maurer-Cartan element that is equivalent to zero, see \cite[Appendix~C]{bursztyn.dolgushev.waldmann:2012a} for this observation and 
\cite{dolgushev:2005a,dolgushev:2005b,dotsenko.shadrin.vallette:2018,
esposito.dekleijn:2018a:pre} for 
more details on the twisting procedure. This corresponds to a more 
general observation: Let $F \colon (\liealg{g},\D,[\argument,\argument]) \rightarrow 
(\liealg{g}',\D',[\argument,\argument])$ be an 
$L_\infty$-morphism between DGLAs with complete descending and 
exhaustive filtrations $\mathcal{F}^\bullet\liealg{g}$ resp. 
$\mathcal{F}^\bullet\liealg{g}'$. Moreover, let $\pi \in \mathcal{F}^1\liealg{g}^1$ 
be a Maurer-Cartan element  equivalent to zero via 
$\pi = \exp([g,\argument])\acts 0$ with $g \in \mathcal{F}^1\liealg{g}^0$. 
The element $\pi' = \sum_{k=1}^\infty\frac{1}{k!} F^1_k(\pi
\vee \cdots \vee \pi) \in \mathcal{F}^1\liealg{g}'^1$ is a 
Maurer-Cartan element in $\liealg{g}'$ equivalent to zero. Let the equivalence be given 
by $g' \in \mathcal{F}^1\liealg{g}'^0$, then one obtains 
Proposition~\ref{prop:TwistMorphHomEqu}:

\begin{nnproposition}
  The $L_\infty$-morphisms $F$ and $e^{[-g',\argument]}\circ F^{\pi} \circ 
	e^{[g,\argument]}$ from $(\liealg{g},\D,[\argument,\argument])$ to 
	$(\liealg{g}',\D',[\argument,\argument])$ are homotopic, where 
	$F^\pi$ denotes the $L_\infty$-morphism $F$ twisted by $\pi$. 
\end{nnproposition}
By homotopic we mean here that the two $L_\infty$-morphisms are equivalent 
Maurer-Cartan elements in the convolution DGLA, compare 
\cite[Definition~3]{dolgushev:2007a}, see also \cite{dotsenko.poncin:2016a} for a 
comparison of different notions of homotopies between $L_\infty$-morphisms. 

This general statement can be applied to the globalization of the Kontsevich formality.
Our main result here is the following theorem, see 
Theorem~\ref{thm:GlobalwithModifareHomotopic}: 

\begin{nntheorem}
  Let $\nabla$ and $\nabla'$ be two different torsion-free covariant 
	derivatives. Then the two global formalities constructed via Dolgushev's globalization 
	procedure are homotopic.
\end{nntheorem}
This immediately implies that they induce the same map on the equivalence classes of 
formal Maurer-Cartan elements, i.e. \cite[Theorem~2.6]{bursztyn.dolgushev.waldmann:2012a}.

Note that there are many other similar globalization procedures of 
formalities based on Dolgushev's globalization of the Kontsevich formality 
\cite{dolgushev:2005a,dolgushev:2005b}, e.g. \cite{calaque:2005a} for Lie algebroids, 
\cite{liao.stienon.xu:2018a} for differential graded manifolds and 
\cite{dolgushev:2006a} for Hochschild chains. The 
above technique can be adapted to these cases and we plan to  pursue them 
in further works.

Finally, we want to mention that in \cite[Section~7]{kontsevich:2003a} 
there is also 
a globalization procedure explained, using the language of $\infty$-jet 
spaces of polyvector fields and polydifferential operators, respectively. 
However, these $\infty$-jet spaces are (non-canonically) isomorphic as 
vector bundles to the formally completed fiberwise polyvector fields and 
polydifferential operators, respectively.
The corresponding  isomorphisms are constructed by the choice of a 
connnection. We strongly believe that the globalization procedure proposed 
by Kontsevich in \cite{kontsevich:2003a} is homotopic to the 
globalization from Dolgushev \cite{dolgushev:2005a,dolgushev:2005b} 
we are using in this note. 

The paper is organized as follows: In Section~\ref{subsec:MCandEquiv} we recall 
the basics concerning Maurer-Cartan elements in DGLAs and $L_\infty$-algebras, 
the notions of gauge and homotopy equivalence as well as the twisting procedure.
Then we recall in Section~\ref{section:RelbetTwistedMorphisms} the 
interpretation of $L_\infty$-morphisms as Maurer-Cartan elements and the notion 
of homotopic $L_\infty$-morphisms. We show that pre- and post-compositions of 
homotopic $L_\infty$-morphisms with an $L_\infty$-morphism are again homotopic, 
a statement that is probably well-known to the experts, but that we could not 
find in the literature. Moreover, we prove here Proposition~\ref{prop:TwistMorphHomEqu}, 
i.e. that the twisted $L_\infty$-morphisms are homotopic for equivalent 
Maurer-Cartan elements. Finally, we apply these general results to the globalization 
of Kontsevich's formality theorem, proving Theorem~\ref{thm:GlobalwithModifareHomotopic} 
and also an equivariant version for Lie group actions with invariant covariant 
derivatives.

\textbf{Acknowledgements:}
  The authors are grateful to Chiara Esposito, 
	Ryszard Nest and Boris Tsygan for the idea 
	leading to this letter and for many helpful comments. 
	This work was supported by the National Group for Algebraic and 
	Geometric Structures, and their Applications (GNSAGA – INdAM).
	The second author is supported by the DFG research training 
	group "gk1821: Cohomological Methods in Geometry".

\section{Preliminaries: Maurer-Cartan Elements and Twisting}
\label{subsec:MCandEquiv}

\subsection{Maurer-Cartan Elements in DGLAs}

We want to recall the basics concerning differential graded Lie algebras 
(DGLAs), Maurer-Cartan elements and their equivalence classes. In order to 
make sense of the gauge equivalence we consider in this context 
DGLAs $(\liealg{g}^\bullet,\D,[\argument,\argument])$ with complete descending filtrations 
\begin{equation}
  \cdots 
	\supseteq 
	\mathcal{F}^{-2}\liealg{g}
	\supseteq 
	\mathcal{F}^{-1}\liealg{g}
	\supseteq 
	\mathcal{F}^{0}\liealg{g}
	\supseteq 
	\mathcal{F}^{1}\liealg{g}
	\supseteq 
	\cdots,
	\quad \quad
	\liealg{g}
	\cong
	\varprojlim \liealg{g}/\mathcal{F}^n\liealg{g}
\end{equation}  
and 
\begin{equation}
  \D(\mathcal{F}^k\liealg{g})
	\subseteq
	\mathcal{F}^k\liealg{g}
	\quad \quad \text{ and } \quad \quad
	[\mathcal{F}^k\liealg{g},\mathcal{F}^\ell\liealg{g}]
	\subseteq 
	\mathcal{F}^{k+\ell}\liealg{g}.
\end{equation}
In particular, $\mathcal{F}^1\liealg{g}$ is a projective limit of 
nilpotent DGLAs. In most cases the filtration will be bounded below, i.e. 
bounded from the left with $\liealg{g}=\mathcal{F}^k\liealg{g}$ for some 
$k\in \mathbb{Z}$. If the filtration is unbounded, then we assume always that 
it is in addition exhaustive, i.e. that
\begin{equation}
  \liealg{g}
	=
	\bigcup_n \mathcal{F}^n\liealg{g},
\end{equation}
even if we do not mention it explicitly. Moreover, we assume that the 
DGLA morphisms are compatible with the filtrations.

\begin{example}
  One motivation to consider the case of filtered DGLAs are formal 
  power series $\liealg{g}[[\hbar]]$ of a DGLA $\liealg{g}$ with filtration 
  $\mathcal{F}^k (\liealg{g}[[\hbar]]) = \hbar^k (\liealg{g}[[\hbar]])$.
\end{example}

\begin{definition}[Maurer-Cartan elements]
  Let $(\liealg{g},\D,[\,\cdot\,,\,\cdot\,])$ be a DGLA with complete 
	descending filtration. Then $\pi \in \mathcal{F}^1\liealg{g}^1$ is called 
	\emph{Maurer-Cartan element} if it 
	satisfies the Maurer-Cartan equation
	\begin{equation}
	  \D \pi + \frac{1}{2}[\pi,\pi]
		=
		0.
	\end{equation}
	The set of Maurer-Cartan elements is denoted by $\Mc(\liealg{g})$.
\end{definition}
Maurer-Cartan elements $\pi$ lead to \emph{twisted} DGLA structures 
$(\liealg{g},\D+ [\pi,\argument],[\argument,\argument])$ and one has 
a gauge action on the set of Maurer-Cartan elements.

\begin{proposition}[Gauge action]
  \label{prop:GaugeactionDGLA}
  Let $(\liealg{g},\D,[\,\cdot\,,\,\cdot\,])$ be a DGLA with complete 
	descending filtration. The \emph{gauge group} 
	$ \group{G}^0(\liealg{g})  =  \{ \Phi = e^{[g,\,\cdot\,]} \colon \liealg{g} 
	  \longrightarrow \liealg{g} \mid g \in \mathcal{F}^1\liealg{g}^0\}$		
  defines an action on $\Mc(\liealg{g})$ via
	\begin{equation}
	  \label{eq:gaugeactiong0}
	  \exp([g,\,\cdot\,]) \acts \pi 
		=
		\sum_{n=0}^\infty \frac{( [g,\,\cdot\,])^n}{n!}(\pi) 
		-
		\sum_{n=0}^\infty \frac{([g,\,\cdot\,])^n}{(n+1)!} (\D g)
		=
		\pi - \frac{\exp([g,\,\cdot\,]) - \id}{[g,\argument]}(\D g + [\pi,g]).
	\end{equation}
	The set of equivalence classes of Maurer-Cartan elements in $\liealg{g}$ 
	is denoted by
	\begin{equation}
	  \Def(\liealg{g})
		=
		\frac{\Mc(\liealg{g})} {\group{G}^0(\liealg{g})}.
	\end{equation}
\end{proposition}
Note that the gauge action is well-defined since $g \in \mathcal{F}^1\liealg{g}$ 
and as the filtration is complete. 
$\Def(\liealg{g})$ is the transformation groupoid of the gauge action and 
also called \emph{Goldman-Millson groupoid} \cite{goldman.millson:1988a}. 
It plays an important role in deformation theory \cite{manetti:2005a}.
In particular, the definition implies that twisting with gauge equivalent 
Maurer-Cartan elements leads to isomorphic DGLAs.

\begin{corollary}
  \label{cor:GaugeEquivMCTwistsQuis}
  Let $(\liealg{g},\D,[\,\cdot\,,\,\cdot\,])$ be a DGLA with complete 
	descending filtration and with gauge equivalent Maurer-Carten elements 
	$\pi',\pi$ via $g \in \group{G}^0(\liealg{g})$. Then one has
	\begin{equation}
		\label{eq:MCEquivalenceIso}
		\D + [\pi',\argument]
		=
		\exp([g,\,\cdot\,]) \circ (\D + [\pi,\argument]) \circ \exp([-g,\,\cdot\,]).
	\end{equation}
	In other words, $\exp([g,\argument]) \colon  (\liealg{g},\D+[\pi,\argument],
	[\,\cdot\,,\,\cdot\,]) \rightarrow (\liealg{g},\D+[\pi',\argument],
	[\,\cdot\,,\,\cdot\,])$ is an isomorphism of DGLAs.
\end{corollary}

\subsection{Maurer-Cartan Elements in $L_\infty$-algebras}

Let us recall the basics of $L_\infty$-algebras and $L_\infty$-morphisms. 
Proofs and further details can be found in
\cite{dolgushev:2005a,dolgushev:2005b,esposito.dekleijn:2018a:pre}. Note that in 
this work we only consider $L_\infty$-morphisms between DGLAs.

An \emph{$L_\infty$-algebra} $(L,Q)$ is a graded vector space $L$ together 
with a degree $+1$ codifferential $Q$ on the graded cocommutative cofree 
coalgebra $(\cc{\Sym}(L[1]),\cc{\Delta})$ without counit cogenerated by $L[1]$. 
We always consider a vector space over a field 
$\mathbb{K}$ of characteristic zero. The codifferential $Q$ is uniquely determined by the
Taylor components $ Q_n\colon \Sym^n(L[1])\longrightarrow L[2]$ for $n \geq 1$.
Sometimes we also write $Q_k = Q_k^1$ and following
\cite{canonaco:1999a} we denote by $Q_n^i$ the component of $Q_n^i
\colon \Sym^n (L[1]) \rightarrow \Sym^i (L[1])[1]$ of $Q$. The property $Q^2=0$ 
implies in particular that $Q_1^1\colon L \rightarrow L[1]$ is a 
cochain differential.
Let us consider two $L_\infty$-algebras $(L,Q)$ and
$(L',Q')$.  A degree $0$ coalgebra morphism
$ F\colon  \cc{\Sym}(L[1]) \longrightarrow  \cc{\Sym}(L'[1])$
such that $FQ = Q'F$ is called \emph{$L_\infty$-morphism}. 
Just like the codifferential also the morphism $F$ is also uniquely determined by its
Taylor components 
$  F_n\colon \Sym^n(L[1])\longrightarrow L'[1]$,
where $n\geq 1$. We write again $F_k = F_k^1$ and we get 
coefficients $F_n^j \colon \Sym^n (L[1]) \rightarrow 
\Sym^j (L'[1])$ of $F$. 
Note that $F_n^j$ depends only on $F_k^1 = F_k$ for $k\leq n-j+1$. 
In particular, the first structure map of $F$ is a map of complexes 
$F_1^1\colon (L,Q_1^1) \rightarrow (L',(Q')_1^1)$ and one calls $F$ 
\emph{$L_\infty$-quasi-isomorphism} if $F_1^1$ is a quasi-isomorphism 
of complexes.

\begin{example}[DGLA]
  \label{ex:DGLAasLinfty}
  A DGLA $(\liealg{g},\D,[\,\cdot\,,\,\cdot\,])$ is an $L_\infty$-algebra 
	with $Q_1 = -\D$ and $Q_2(\gamma \vee \mu) = 
	-(-1)^{\abs{\gamma}}[\gamma,\mu]$, where $\abs{\gamma}$ denotes the degree 
	in $\liealg{g}[1]$.
\end{example}

In order to generalize the definition of Maurer-Cartan elements we 
consider again $L_\infty$-algebras with complete descending and exhaustive 
filtrations on $L$. We assume again that $L_\infty$-morphisms are compatible 
with the filtrations. 

\begin{definition}[Maurer-Cartan elements II]
  Let $(L,Q)$ be an $L_\infty$-algebra 
	with compatible complete descending filtration. 
	Then $\pi\in \mathcal{F}^1 L[1]^0$ is called 
	\emph{Maurer-Cartan element} if it satisfies the Maurer-Cartan equation
	\begin{equation}
	  \label{eq:mclinfty}
	  \sum_{n>0} \frac{1}{n!} Q_n(\pi\vee\cdots\vee \pi)
		=
		0.
	\end{equation}
	The set of Maurer-Cartan elements is again denoted by $\Mc(L)$.
\end{definition}

Note that the sum in \eqref{eq:mclinfty} is well-defined for $x\in 
\mathcal{F}^1L^1$ because of the completeness of $L$.
We recall some useful properties from \cite[Prop.~1]{dolgushev:2005b}:

\begin{lemma}
  \label{lemma:twistinglinftymorphisms}
  Let $F\colon (\liealg{g},Q)\rightarrow (\liealg{g}',Q')$ be an 
	$L_\infty$-morphism of DGLAs and $\pi \in \mathcal{F}^1\liealg{g}^1$. 
	\begin{lemmalist}			
		\item $ d \pi + \frac{1}{2}[\pi,\pi]=0$ is equivalent to 
		      $Q(\cc{\exp}(\pi))=0$, where $\cc{\exp}(\pi)=
					\sum_{k=1}^\infty \frac{1}{k!} 
					\pi^{\vee k}$.
		
		\item $F(\cc{\exp}(\pi)) = \cc{\exp}(S)$ with 
		      $S = F^1(\cc{\exp}(\pi))= 
					\sum_{n>0}\frac{1}{n!} F_n(\pi\vee \cdots \vee\pi)$.
					
		\item If $\pi$ is a Maurer-Cartan element, then so is $S$.
  \end{lemmalist}
\end{lemma}

We recall the generalization of the gauge action to an equivalence 
relation on the set of Maurer-Cartan elements of $L_\infty$-algebras. 
We follow \cite[Section~4]{canonaco:1999a} but 
adapt the definitions to the case of $L_\infty$-algebras with complete 
descending and exhaustive filtrations as in \cite{dotsenko.poncin:2016a}. 
Let therefore 
$(L,Q)$ be such an $L_\infty$-algebra with complete descending and 
exhaustive filtration and 
consider $L[t]=L \otimes \mathbb{K}[t]$ which has again a descending 
and exhaustive filtration
\begin{equation*}
  \mathcal{F}^k L[t]
	=
	\mathcal{F}^kL \otimes \mathbb{K}[t].
\end{equation*} 
We denote its completion by $\widehat{L[t]}$ and note that since $Q$ is compatible 
with the filtration it extends to $\widehat{L[t]}$. Similarly, 
$L_\infty$-morphisms extend to these completed spaces. 

\begin{remark} 
  \label{rm:completion}
  Note that one can define the completion as space of equivalence classes of 
	Cauchy sequences with respect to the filtration topology.
  Alternatively, the completion can be identified with
  \begin{equation*}
    \varprojlim L[t] / \mathcal{F}^n L[t] 
	  \subset \prod_n L[t]/\mathcal{F}^nL[t]
		\cong
		\prod_n L/\mathcal{F}^nL \otimes\mathbb{K}[t]
  \end{equation*}
  consisting of all coherent tuples $X=(x_n)_n \in 
	\prod_n L[t]/\mathcal{F}^nL[t] $, where
  \begin{equation*}
    L[t]/\mathcal{F}^{n+1}L[t] \ni x_{n+1}
	  \longmapsto
	  x_n \in L[t]/\mathcal{F}^n[t]
  \end{equation*}
	under the obvious surjections. 
	Moreover, $\mathcal{F}^n\widehat{L[t]}$ corresponds to the kernel 
	of $\varprojlim L[t]/\mathcal{F}^nL[t] \rightarrow L[t]/\mathcal{F}^nL[t]$ 
	and thus
	\begin{equation*}
	  \widehat{L[t]} / \mathcal{F}^n\widehat{L[t]} 
		\cong 
    L[t] / \mathcal{F}^n L[t].
	\end{equation*}
	Since $L$ is complete, we can also interpret $\widehat{L[t]}$ as the 
	subspace of $L[[t]]$ such that 
	$X\, \mathrm{ mod } \,\mathcal{F}^nL[[t]]$ is polynomial in $t$. In particular, 
	$\mathcal{F}^n\widehat{L[t]}$ is the subspace of elements in 
	$\mathcal{F}^nL[[t]]$ that are polynomial in $t$ modulo 
	$\mathcal{F}^mL[[t]]$ for all $m$.
\end{remark}

By the above construction of $\widehat{L[t]}$ it is clear that 
differentiation $\frac{\D}{\D t}$ and integration with respect to 
$t$ extend to it since they do not change the filtration. Sometimes we 
write also $\dot{X}$ instead of $\frac{\D}{\D t}X$ and, 
moreover, the evaluation
\begin{equation*}
  \delta_s 
	\colon
	\widehat{L[t]} \ni X
	\longmapsto
	X(s)
	=
	X\at{t=s} \in L
\end{equation*}
is well-defined for all $s\in \mathbb{K}$ since $L$ is complete.

\begin{example}
  In the case that the filtration of $L$ comes from a grading $L^\bullet$, the 
	completion is given by $\widehat{L[t]}\cong \prod_i L^i[t]$, i.e. by polynomials 
	in each degree. A special case is here the case of formal power series 
	$L= V[[\hbar]]$ with $\widehat{L[t]} \cong (V[t])[[\hbar]]$ as in 
	\cite[Appendix~A]{bursztyn.dolgushev.waldmann:2012a}.
\end{example}

Now we can introduce a general equivalence relation between Maurer-Cartan 
elements of $L_\infty$-algebras.

\begin{definition}[Homotopy equivalence]
  Let $(L,Q)$ be a $L_\infty$-algebra with a complete descending filtration. 
	The \emph{homotopy equivalence relation} on the set $\Mc(L)$ is 
	the transitive closure 
	of the relation $\sim$ defined by: $\pi_0 \sim \pi_1$ if and only if 
	there exist $\pi(t) \in \mathcal{F}^1\widehat{L^1[t]}$ and 
	$\lambda(t) \in \mathcal{F}^1\widehat{L^0[t]}$ such that
	\begin{align}
	  \label{eq:EquivMCElements}
		\begin{split}
		  \frac{\D}{\D t} \pi(t)
			& = 
			Q^1  (\lambda(t) \vee \exp(\pi(t)))
			=
			\sum_{n=0}^\infty \frac{1}{n!} Q^1_{n+1} (\lambda(t) \vee \pi(t) \vee 
			\cdots \vee \pi(t)),   \\
			\pi(0) 
			& = 
			\pi_0 
			\quad \quad \text{ and }\quad \quad 
			\pi(1)
			=
			\pi_1.
		\end{split}
	\end{align}
	The set of equivalence classes of Maurer-Cartan elements of $L$ is 
	denoted by $\Def(L) = \Mc(L) / \sim$.
\end{definition}

Note that in the case of 
nilpotent $L_\infty$-algebras it suffices to consider polynomials in $t$ 
as there is no need to complete $L[t]$, compare \cite{getzler:2009a}.
We check now that this is well-defined and even yields a curve $\pi(t)$ 
of Maurer-Cartan elements.

\begin{proposition}
  \label{prop:PioftUnique}
  For every $\pi_0 \in \mathcal{F}^1L^1$ and $\lambda(t) \in 
	\mathcal{F}^1\widehat{L^0[t]}$ 
	there exists a unique $\pi(t) \in \mathcal{F}^1\widehat{L^1[t]}$ 
	such that 
	$\frac{\D}{\D t} \pi(t) = Q^1  (\lambda(t) \vee \exp(\pi(t)))$ and 
	$\pi(0) = \pi_0$. If $\pi_0 \in \Mc(L)$, then $\pi(s) \in \Mc(L)$ for 
	all $s\in \mathbb{K}$.
\end{proposition}
\begin{proof}
The proof for the nilpotent case can be found in 
\cite[Prop.~4.8]{canonaco:1999a}. In our setting of complete filtrations 
we only have to show that the solution $\pi(t) = \sum_{k=0}^\infty 
\pi_k t^k$ in the formal power series 
$\mathcal{F}^1 L^1 \otimes \mathbb{K}[[t]]$
is an element of $ \mathcal{F}^1\widehat{L^1[t]}$. 
By Remark~\ref{rm:completion} this is equivalent to 
$\pi(t) \,\mathrm{mod}\, \mathcal{F}^nL^1[[t]] \in L^1[t]$ for all $n$. 
Indeed, we have inductively
\begin{equation*}
  \frac{\D}{\D t}\pi(t) \mod \mathcal{F}^2L^1[[t]]
	=
	Q^1(\lambda(t)) \mod \mathcal{F}^2L^1[[t] 
	\in 
	L^1[1].
\end{equation*}
For the higher orders we get
\begin{equation*}
  \frac{\D}{\D t}\pi(t) \mod \mathcal{F}^nL^1[[t]]
	=
	\sum_{k=0}^{n-1}\frac{1}{k!} 
	Q^1_k(\lambda(t)\vee (\pi(t)+\mathcal{F}^{n-1})\vee \cdots 
	\vee (\pi(t)+\mathcal{F}^{n-1}))
	\mod \mathcal{F}^nL^1[[t]]
\end{equation*}
and thus $\pi(t) \,\mathrm{mod}\, \mathcal{F}^nL^1[[t]] \in L^1[t]$.
\end{proof}

One can show that for DGLAs with complete filtrations the two notions 
of equivalences are equivalent, 
see e.g. \cite[Thm.~5.5]{manetti:2005a}.

\begin{theorem}
  \label{thm:DGALHomvsGaugeEquiv}
  Two Maurer-Cartan elements in $(\liealg{g},\D,[\argument,\argument])$ 
	are homotopy equivalent if and only if they are gauge equivalent.
\end{theorem}

This theorem can be rephrased in a more explicit manner in the following proposition.

\begin{proposition}
\label{prop:HomEquvsGaugeEqu}
  Let $(\liealg{g},\D,[\argument,\argument])$ be a DGLA 
	with complete descending filtration. 
	Consider $\pi_0 \sim \pi_1$ with equivalence given by  
  $\pi(t) \in \mathcal{F}^1\widehat{\liealg{g}^1[t]}$ and 
	$\lambda(t) \in \mathcal{F}^1\widehat{ \liealg{g}^0[t]}$. 
	The formal solution of
	\begin{equation}
		\label{eq:ODEforA}
		\lambda(t)
		=
		\frac{\exp([A(t),\argument])-\id}{[A(t),\argument]} \frac{\D A(t)}{\D t} ,
		\quad\quad
		A(0)
		=
		0
	\end{equation}
	is an element $A(t) \in \mathcal{F}^1\widehat{ \liealg{g}^0[t]}$
	and satisfies
	\begin{equation}
	  \pi(t)
		=
		e^{[A(t),\argument]}\pi_0 
	  - \frac{\exp([A(t),\argument]-\id}{[A(t),\argument]} \D A(t).
	\end{equation}
	In particular, for $g =A(1) \in  \mathcal{F}^1 \liealg{g}^0$ one has 
	\begin{equation}
	  \pi_1
		=
		\exp([g,\argument])\acts \pi_0.
	\end{equation}
\end{proposition}
\begin{proof}
As formal power series in $t$ Equation~\ref{eq:ODEforA} has a unique solution
$A(t) \in \mathcal{F}^1 \liealg{g}^0 \otimes \mathbb{K}[[t]]$.
But one has even $A(t)\in \mathcal{F}^1\widehat{\liealg{g}^0[t]}$ 
since
\begin{align*}
  \frac{\D A(t)}{\D t}  &
	\equiv
	\lambda(t) - \sum_{k=1}^{n-2} \frac{1}{(k+1)!} [A(t),\argument]^k 
	\frac{\D A(t)}{\D t} 
	\mod \mathcal{F}^n\liealg{g}[[t]] \\
	& \equiv
	\lambda(t) - \sum_{k=1}^{n-2} \frac{1}{(k+1)!} 
	[A(t)\mod \mathcal{F}^{n-1}\liealg{g}[[t]],\argument]^k 
	\left(\frac{\D A(t)}{\D t} \mod \mathcal{F}^{n-1}\liealg{g}[[t]]\right)
	\mod \mathcal{F}^n\liealg{g}[[t]] 
\end{align*}
is by induction polynomial in $t$.
Note that one has
\begin{equation}
  \tag{$*$}
  \label{eq:DiffofExp}
  \frac{\D}{\D t} e^{[A(t),\argument]}
	=
	\left[\frac{\exp([A(t),\argument]-\id}{[A(t),\argument]}\frac{\D A(t)}{\D t}  , 
	\argument \right] 
	\circ
	\exp([A(t),\argument]).
\end{equation}
Our aim is now to show that 
$\pi'(t) = e^{[A(t),\argument]}\pi_0 
	- \frac{\exp([A(t),\argument]-\id}{[A(t),\argument]} \D A(t)$ 
satisfies
\begin{equation*}
  \frac{\D \pi'(t)}{\D t}
  =
  -\D \lambda(t) + \left[\lambda(t), e^{[A(t),\argument]}\pi_0 
	- \frac{\exp([A(t),\argument])-\id}{[A(t),\argument]} \D A(t)\right].
\end{equation*} 
Then we know $\pi'(t) = \pi(t)\in \mathcal{F}^1\widehat{\liealg{g}^1[t]}$ 
since the solution $\pi(t)$ is unique by 
Proposition~\ref{prop:PioftUnique}, which immediately gives 
$\pi'(1)=\pi_1$. At first we compute
\begin{align*}
  \D \lambda(t)
	& =
	\frac{\exp([A(t),\argument])-\id}{[A(t),\argument]}\D \frac{\D A(t)}{\D t}
	+
	\sum_{k=0}^\infty \sum_{j=0}^{k-1} \frac{1}{(k+1)!} \binom{k}{j+1} 
	\left[ \ad_A^j\D A(t), \ad_A^{k-1-j}  \frac{\D A(t)}{\D t}\right]
\end{align*} 
and using \eqref{eq:DiffofExp} we get
\begin{align*}
  \frac{\D \pi'(t)}{\D t}
	& =
	\left[\frac{\exp([A(t),\argument])-\id}{[A(t),\argument]} \frac{\D A(t)}{\D t}, 
	\exp([A(t),\argument]) \pi_0\right] - 
	\frac{\exp([A(t),\argument])-\id}{[A(t),\argument]}\D \frac{\D A(t)}{\D t}  \\
	& \;
	- \sum_{k=0}^\infty \sum_{j=0}^{k-1} \frac{1}{(k+1)!}\binom{k}{j+1} 
	\left[\ad_A^j \frac{\D A(t)}{\D t}, \ad_A^{k-1-j}\D A\right]  \\
	& =
	-\D \lambda(t) + \left[\lambda(t), e^{[A(t),\argument]}\pi_0 
	- \frac{\exp([A(t),\argument])-\id}{[A(t),\argument]} \D A(t)\right]
\end{align*}
and the proposition is proven.
\end{proof}

\begin{remark}
  There are also different notions of homotopy resp. gauge 
	equivalences for Maurer-Cartan elements in $L_\infty$-algebras: 
	e.g. the above definition, sometimes also called \emph{Quillen homotopy}, and 
	the \emph{gauge homotopy} where one requires $\lambda(t) = \lambda$ to be 
	constant, compare \cite{dolgushev:2007a}. 
	In \cite{dotsenko.poncin:2016a} 
	it is shown that these notions are also equivalent for complete 
	$L_\infty$-algebras, extending the
	result for DGLAs. 
\end{remark}

One important property is that $L_\infty$-morphisms map equivalence classes of 
Maurer-Cartan elements to equivalence classes, 
see \cite[Prop.~4.9]{canonaco:1999a}.

\begin{proposition}
\label{prop:FmapsEquivMCtoEquiv}
  Let $F \colon (L,Q) \rightarrow (L',Q')$ be an $L_\infty$-morphism between 
	$L_\infty$-algebras with complete filtrations, 
	and $\pi_0,\pi_1 \in \Mc(L)$ with $\pi_0\sim\pi_1$ via 
	 $\pi(t) \in \mathcal{F}^1\widehat{\liealg{g}^1[t]}$ and 
	$\lambda(t) \in \mathcal{F}^1\widehat{\liealg{g}^0[t]}$. 
	Then $F$ is compatible with the homotopy equivalence relation, i.e. one has 
	$F^1(\cc{\exp}\pi_0)\sim F^1(\cc{\exp}\pi_1)$ via 
	\begin{equation*}
	  \pi'(t) 
		= 
		F^1(\cc{\exp}(\pi(t))) 
		\quad \text{ and } \quad
		\lambda'(t) 
		= 
		F^{1}(\lambda(t)\vee \exp(\pi(t))).
	\end{equation*}
\end{proposition}

If $F$ is an $L_\infty$-quasi-isomorphism, then it is well-known that it 
induces a bijection on the equivalence classes of 
Maurer-Cartan elements. Finally, recall that also the twisting with 
Maurer-Cartan elements 
can be generalized to $L_\infty$-algebras, see e.g. 
\cite[Section~2.3]{dolgushev:2006a}.

\begin{lemma} 
Let $(L,Q)$ be an $L_\infty$-algebra and $\pi \in \mathcal{F}^1 L[1]^0$ 
a Maurer-Cartan element. Then the map $Q^\pi$ given by
\begin{equation}
  Q^\pi(X)
	=
	\exp((-\pi)\vee) Q (\exp(\pi \vee)X),
	\quad \quad
	X \in \cc{\Sym}(L[1])
\end{equation}
defines a codifferential on $\cc{\Sym}(L[1])$.
\end{lemma}

One can not only twist the DGLAs resp. $L_\infty$-algebras, but also the 
$L_\infty$-morphisms between them. Below we need the following result, see  
\cite[Prop.~2]{dolgushev:2006a} and \cite[Prop.~1]{dolgushev:2005b}.

\begin{proposition}
  \label{prop:twistinglinftymorphisms}
  Let $F\colon (\liealg{g},Q)\rightarrow (\liealg{g}',Q')$ be an 
	$L_\infty$-morphism of DGLAs, $\pi \in \mathcal{F}^1\liealg{g}^1$ 
	a Maurer-Cartan element and $S = F^1(\cc{\exp}(\pi))
	\in \mathcal{F}^1 \liealg{g}'^1$.
	\begin{propositionlist}					
		\item The map 
		      \begin{equation*}
					  F^\pi
						=
						\exp(-S\vee) F \exp(\pi\vee) \colon 
						\cc{\Sym}(\liealg{g}[1])
						\longrightarrow
						\cc{\Sym}(\liealg{g}'[1])
					\end{equation*}
					defines an $L_\infty$-morphism between the DGLAs 
					$(\liealg{g},\D+[\pi,\,\cdot\,])$ and $(\liealg{g}',\D+[S,\,\cdot\,])$.
					
		\item The structure maps of $F^\pi$ are given by 
		      \begin{equation}
					  \label{eq:twisteslinftymorphism}
					  F_n^\pi(x_1,\dots, x_n)
						=
						\sum_{k=0}^\infty \frac{1}{k!} 
						F_{n+k}(\pi, \dots, \pi,x_1 ,	\dots, x_n).
					\end{equation}
			
		\item Let $F$ be an $L_\infty$-quasi-isomorphism such that 
		      $F_1^1$ is not only a quasi-isomorphism of filtered complexes 
					$L\rightarrow L'$ but even induces a quasi-isomorphism
          \begin{equation*}
            F_1^1 \colon
	          \mathcal{F}^k L
	          \longrightarrow
	          \mathcal{F}^kL'
          \end{equation*}
          for each $k$. Then $F^\pi$ is an $L_\infty$-quasi-isomorphism.	 
	\end{propositionlist}
\end{proposition}

\section{Relation between Twisted Morphisms}
\label{section:RelbetTwistedMorphisms}

Here we prove the main results about the relation 
between twisted $L_\infty$-morphisms. 
More explicitly, consider an $L_\infty$-morphism
$F \colon (\liealg{g},Q) \rightarrow (\liealg{g}',Q')$ 
between DGLAs and let $\pi_0,\pi_1 \in \mathcal{F}^1\liealg{g}^1$ be two 
equivalent Maurer-Cartan elements via 
$\pi_1 = \exp([g,\argument])\acts \pi_0$. We show that 
$F^{\pi_0}$ and $F^{\pi_1}$ can be interpreted as homotopic in the sense of 
\cite[Definition~3]{dolgushev:2007a}. 

\subsection{$L_\infty$-morphisms as Maurer-Cartan Elements}
At first, recall that 
we can interpret $L_\infty$-morphisms as Maurer-Cartan elements in 
the convolution algebra. More explicitly, let $(L,Q),(L',Q')$ be two 
$L_\infty$-algebras and denote the graded linear maps 
by $\Hom(\cc{\Sym}(L[1]),L')$. If $L$ and $L'$ are equipped with complete descending 
filtrations, then we require the maps to be compatible with the filtration. The 
$L_\infty$-structures on $L$ and $L'$ lead to an $L_\infty$-structure on this 
vector space of maps, see \cite[Proposition~1 and Proposition~2]{dolgushev:2007a} 
and also \cite{bursztyn.dolgushev.waldmann:2012a} for the case of DGLAs.

\begin{proposition}
	The coalgebra $\cc{\Sym}(\Hom(\cc{\Sym}(L[1]),L')[1])$ can be equipped with 
	a codifferential $\widehat{Q}$ with structure maps
  \begin{equation}
	  \label{eq:DiffonConvLinfty}
	  \widehat{Q}^1_1 F
		=
		Q'^1_1 \circ F - (-1)^{\abs{F}} F \circ Q
	\end{equation}
	and 
	\begin{equation}
		\label{eq:BracketonConvLinfty}
		\widehat{Q}^1_n(F_1\vee \cdots \vee F_n)
		=
		(Q')^1_n \circ \vee^{n-1}\circ
		(F_1\otimes F_2\otimes \cdots \otimes F_n) \circ \cc{\Delta}^{n-1}.
	\end{equation}
	It is called \emph{convolution $L_\infty$-algebra} and its 
	Maurer-Cartan elements are identified with $L_\infty$-morphisms. 
	Here $\abs{F}$ denotes the degree in 
	$\Hom(\cc{\Sym}(L[1]),L')[1]$. 
\end{proposition}

\begin{example}
  Let $\liealg{g},\liealg{g}'$ be two DGLAs. Then 
	$\Hom(\cc{\Sym}(\liealg{g}[1]),	\liealg{g}')$ is in fact a DGLA 
	with differential
	\begin{equation}
	  \label{eq:DiffonConvDGLA}
	  \del F
		=
		\D' \circ F + (-1)^{\abs{F}} F \circ Q
	\end{equation}
	and Lie bracket
	\begin{equation}
		\label{eq:BracketonConvDGLA}
		[F,G]
		=		
		- (-1)^{\abs{F}} (Q')^1_2 \circ (F\otimes G) \circ \cc{\Delta}.
	\end{equation}
	Here $\abs{F}$ denotes again the degree in 
	$\Hom(\cc{\Sym}(\liealg{g}[1]),\liealg{g}')[1]$. This DGLA is also called 
	\emph{convolution DGLA}.
\end{example}

We note that the convolution $L_\infty$-algebra 
$\mathcal{H} = \Hom(\cc{\Sym}(L[1]),L')$ is equipped with the following 
complete descending filtration:
\begin{align}
  \label{eq:FiltrationConvLieAlg}
  \begin{split}
  \mathcal{H}
	=
	& \mathcal{F}^1\mathcal{H}
	\supset
	\mathcal{F}^2\mathcal{H}
	\supset \cdots \supset
	\mathcal{F}^k\mathcal{H}
	\supset \cdots \\
	\mathcal{F}^k\mathcal{H}
	& =
	\left\{ f \in \Hom(\cc{\Sym}(L[1]),L') \mid 
	f \at{\Sym^{<k}(L[1])}=0\right\}.
	\end{split}
\end{align}
Thus all twisting procedures are well-defined and 
one can define a notion of homotopic $L_\infty$-morphisms.

\begin{definition}
  \label{def:homotopicMorph}
  Two $L_\infty$-morphisms $F,F'$ from $(L,Q)$ to $(L',Q')$ 
	are called \emph{homotopic} if they are homotopy equivalent Maurer-Cartan 
	elements in the convolution $L_\infty$-algebra $\mathcal{H}$. 
\end{definition}

We collect a few immediate consequences: 

\begin{proposition}
\label{prop:PropertiesofHomotopicMorphisms}
  Let $F,F'$ be two homotopic $L_\infty$-morphisms 
	from $(L,Q)$ to $(L',Q')$.
	\begin{propositionlist}
		\item $F_1^1$ and $(F')_1^1$ are chain homotopic.
		\item If $F$ is an $L_\infty$-quasi-isomorphism, then so is $F'$.
		\item If $L=\liealg{g},L'=\liealg{g}'$ are two DGLAs equipped with 
		      complete descending 
		      filtrations, then $F$ and $F'$ induce the same maps from 
					$\Def(\liealg{g})$ to $\Def(\liealg{g}')$.
		\item In the case of DGLAs $\liealg{g},\liealg{g}'$, compositions 
		      of homotopic $L_\infty$-morphisms with a 
		      DGLA morphism of degree zero are again homotopic.
	\end{propositionlist}
\end{proposition}
\begin{proof}
The first three points are proven in \cite{bursztyn.dolgushev.waldmann:2012a} 
and the last one follows directly.
\end{proof}

We now aim to generalize the last point of the previous proposition to compositions with 
$L_\infty$-morphisms. We start with the post-composition:

\begin{proposition}
  \label{prop:CompofHomotopicHomotopic}
  Let $F_0,F_1$ be two homotopic $L_\infty$-morphisms 
	from $(L,Q)$ to $(L',Q')$. Let 
	$H$ be an $L_\infty$-morphism from $(L',Q')$ to 
	$(L'',Q'')$, then $HF_0 \sim HF_1$. 
\end{proposition}
\begin{proof}
For $F\in \Hom(\cc{\Sym}(L[1]),L')$ we define $\widehat{H}(F)$ via
\begin{equation*}
  (\widehat{H}(F))_n
	=
  (HF)^1_n
	=
	\sum_{\ell=1}^n
	H^1_\ell F^\ell_n
	=
	H^1_\ell \left(\frac{1}{\ell !} F^1\vee \cdots \vee F^1\right) \circ 
	\cc{\Delta}^{\ell -1}.
\end{equation*}
Here the $\vee$-product of maps is given by 
$F\vee G = \vee \circ (F\otimes G) \colon \cc{\Sym}(L[1])\otimes 
\cc{\Sym}(L[1]) \rightarrow \cc{\Sym}(L'[1])$. 
Writing $\cc{\Delta}^\bullet = \sum_{k=0}^\infty \cc{\Delta}^k$ and 
defining all maps to be zero on the domains on which they 
where previously not defined, we can rewrite this as
\begin{equation*}
  \widehat{H}F
	=
	H^1 \circ \cc{\exp} F \circ \cc{\Delta}^\bullet.
\end{equation*}
Let $F(t) \in \widehat{(\Hom(\cc{\Sym}
(L[1]),L')[1])^0[t]}$ and $\lambda(t) \in 
\widehat{(\Hom(\cc{\Sym}(L[1]),L')[1])^{-1}[t]}$ 
describe the homotopy equivalence between $F_0$ and $F_1$. Then 
$\widehat{H}F(t) \in 
\widehat{(\Hom(\cc{\Sym}(L[1]),L'')[1])^{-1}[t]}$ 
satisfies
\begin{align*}
  \frac{\D}{\D t} \widehat{H}F(t)
	& =
	\sum_{\ell =1}^\infty H^1_\ell \frac{\D}{\D t} 
	\left(\frac{1}{\ell !} F(t) \vee \cdots \vee F(t)\right) \circ 
	\cc{\Delta}^{l-1}  \\
	& =
	H^1 \circ \left( 
	\widehat{Q}^1(\lambda(t)\vee \exp(F(t))\vee \exp(F(t)) \right) \circ 
	\cc{\Delta}^\bullet.
\end{align*}
As in \cite[Lemma~4.1]{canonaco:1999a} one can check 
\begin{align*}
  \widehat{Q}(\lambda(t) \vee \exp(F(t)))
	& = 
	\exp(F(t)) \vee \widehat{Q}^1(\lambda(t)\vee \exp(F(t)))
	- \lambda(t)\vee \exp(F(t)) \vee \widehat{Q}^1(\exp(F(t))) \\
	& =
	\exp(F(t)) \vee \widehat{Q}^1(\lambda(t)\vee \exp(F(t)))
\end{align*}
since $F(t)$ is a Maurer-Cartan element.
This allows us to compute
\begin{align*}
  \frac{\D}{\D t} \widehat{H}F(t)
	& =
	H^1 \circ \left( 
	\widehat{Q}(\lambda(t) \vee \exp(F(t))) \right) \circ 
	\cc{\Delta}^\bullet \\\
	& =
	H^1 \circ Q' \circ
	(\lambda(t) \vee \exp(F(t)))  \circ 
	\cc{\Delta}^\bullet 
	+
	H^1\circ(\lambda(t) \vee \exp(F(t)))  \circ 
	\cc{\Delta}^\bullet \circ Q  \\
	& =
	(Q'')^1 \circ H \circ
	(\lambda(t) \vee \exp(F(t)))  \circ 
	\cc{\Delta}^\bullet 
	+
	H^1\circ(\lambda(t) \vee \exp(F(t)))  \circ 
	\cc{\Delta}^\bullet \circ Q  \\
	& =
	(\widehat{Q}')^1_1 \left( H^1 \circ
	(\lambda(t) \vee \exp(F(t))) 
	\circ 
	\cc{\Delta}^\bullet \right) 
	+ \sum_{\ell=2}^\infty (Q'')^1_\ell \circ H^\ell \circ
	(\lambda(t) \vee \exp(F(t)))  \circ 
	\cc{\Delta}^\bullet  .
\end{align*}
Concerning the last term we have omitting the $t$-dependency since $F$ and 
$H$ are of degree zero
\begin{align*}
  & \frac{1}{k!} H^\ell_{k+1}  \circ
	(\lambda(t) \vee F(t)\vee \cdots \vee F(t))  \circ 
	\cc{\Delta}^k (X)\\
	& =
	\frac{1}{k!\ell!} 
	(H^1 \vee \cdots \vee H^1)\circ \cc{\Delta}^{\ell -1}
	\circ
	(\lambda(t) \vee F(t)\vee \cdots \vee F(t))  \circ 
	\cc{\Delta}^k (X) \\
  & =
	\frac{1}{k!\ell!} 
	(H^1 \vee \cdots \vee H^1)\circ
	\sum_{\substack{i_1 + \cdots + i_\ell=k+1 \\ i_j\geq 1}} 
		\sum_{\substack{\sigma\in Sh(i_1,\dots,i_\ell)}}
	\sigma \racts \left( 
	(\lambda(t) \vee F(t)\vee \cdots \vee F(t))  \circ 
	\cc{\Delta}^k	(X)\right)  \\
	& =
	\frac{\ell}{k!\ell!} 
	(H^1 \vee \cdots \vee H^1)\circ 
	\sum_{\substack{i_1 + \cdots + i_\ell=k+1 \\ i_j\geq 1}} 
		\sum_{\substack{\sigma\in Sh(i_1,\dots,i_\ell) \\ \sigma(1)=1}}
	\sigma \racts \left( 
	(\lambda(t) \vee F(t)\vee \cdots \vee F(t))  \circ 
	\cc{\Delta}^k	(X)\right)  \\
	& =
	\frac{1}{(\ell-1)!}
	\sum_{i_1 + \cdots + i_\ell=k+1, i_j\geq 1} 
	\Big(
	\frac{1}{(i_1-1)!} H^1_{i_1}(\lambda \vee F\cdots\vee F) \circ 
	\cc{\Delta}^{i_1-1} \vee 
	\frac{1}{i_2!} 
	H^1_{i_2}(F \vee \cdots\vee F) \circ 
	\cc{\Delta}^{i_2-1}\\
	& \quad \vee\cdots   \vee 
	\frac{1}{i_\ell !} 
	H^1_{i_\ell}(F \vee \cdots\vee F) \circ 
	\cc{\Delta}^{i_\ell-1} 
	\Big) \circ \cc{\Delta}^{\ell -1} (X).
\end{align*}
Here we wrote
\begin{equation*}
  \sigma \racts (x_1\vee \cdots \vee x_{k+1})
	=
	\epsilon(\sigma)
	x_{\sigma(1)}\vee \dots\vee x_{\sigma(i_1)} \otimes  
	 \cdots  \otimes x_{\sigma(k+1-i_\ell+1)}\vee \cdots \vee 
	x_{\sigma(n)}
\end{equation*}
with Koszul sign $\epsilon(\sigma)$. Therefore, it follows
\begin{align*}
  \frac{\D}{\D t} \widehat{H}F(t)
	& =
	(\widehat{Q}')^1_1 \left( H^1 \circ
	(\lambda(t) \vee \exp(F(t))) 
	\circ 
	\cc{\Delta}^\bullet \right) \\ 
	& \quad \quad+ \sum_{\ell=2}^\infty(\widehat{Q}')^1_\ell \circ 
	\left((H^1\circ(\lambda(t)\vee \exp F)\circ \cc{\Delta}^\bullet) \vee 
	\exp(\widehat{H}F)\right)
\end{align*}
and the statement is shown.
\end{proof}

Analogously, we have for the pre-composition:

\begin{proposition}
  \label{prop:preCompofHomotopicHomotopic}
  Let $F_0,F_1$ be two homotopic $L_\infty$-morphisms 
	from $(L,Q)$ to $(L',Q')$. Let 
	$H$ be an $L_\infty$-morphism from $(L'',Q'')$ to 
	$(L,Q)$, then $F_0 H\sim F_1H$. 
\end{proposition}
\begin{proof}
Let $F(t) \in \widehat{(\Hom(\cc{\Sym}
(L[1]),L')[1])^0[t]}$ and $\lambda(t) \in 
\widehat{(\Hom(\cc{\Sym}(L[1]),L')[1])^{-1}[t]}$ 
describe the homotopy equivalence between $F_0$ and $F_1$. Then 
we consider 
\begin{equation*}
  (F(t)H)
	=
	F(t) \circ H
	=
	F(t) \circ \cc{\exp} H^1 \circ \cc{\Delta}^\bullet
	\in
	\widehat{(\Hom(\cc{\Sym}
  (L''[1]),L')[1])^0[t]}
\end{equation*}
in the notation of the above proposition. We compute
\begin{align*}
  \frac{\D}{\D t} (F(t)H)
	& =
	\widehat{Q}^1(\lambda(t)\vee \exp(F(t))) \circ H \\
	& =
	(Q')^1_1 \circ \lambda \circ H 
	+ \lambda \circ Q \circ H
	+ \sum_{\ell=2}^\infty\frac{1}{(\ell-1)!} 
	(Q')^1_\ell \circ (\lambda\vee F\vee \cdots \vee F)
	\circ \cc{\Delta}^{\ell -1}\circ H \\
	& =
	(Q')^1_1 \circ \lambda \circ H 
	+ \lambda \circ H\circ Q'' 
	+ \sum_{\ell=2}^\infty\frac{1}{(\ell-1)!} 
	(Q')^1_\ell \circ (\lambda H\vee F H\vee \cdots \vee F H)
	\circ \cc{\Delta}^{\ell -1}\\
	& =
	\widehat{Q}^1 (\lambda(t)H\vee \exp(F(t)H))
\end{align*}
since $H$ is a coalgebra morphism intertwining $Q''$ and $Q$ and of 
degree zero. Finally, since $\lambda(t)H \in 
\widehat{(\Hom(\cc{\Sym}(L''[1]),L')[1])^{-1}[t]}$ the 
statement follows.
\end{proof}

\subsection{Homotopy Classification of $L_\infty$-algebras}

The above considerations allow us to understand better the homotopy classification 
of $L_\infty$-algebras from \cite{canonaco:1999a,kontsevich:2003a}, which will help us 
in the application to the global formality.

\begin{definition}
\label{def:HomEquLinftyAlgs}
Two  $L_\infty$-algebras $(L,Q)$ and $(L',Q')$ are said to be 
\emph{homotopy equivalent} if there are $L_\infty$-morphisms 
$F\colon (L,Q)\to (L',Q')$ and 
$G\colon (L',Q')\to (L,Q)$ 
such that $F\circ G\sim \id_{L'}$ and 
$G\circ F\sim \id_{L}$. In such case $F$ and $G$ are said
to be quasi-inverse to each other.
\end{definition}

This definition coincides indeed with the definition of homotopy equivalence 
via $L_\infty$-quasi-isomorphisms from \cite{canonaco:1999a}.

\begin{lemma}
\label{lemma:HomEquvsQuasiIso}
Two $L_\infty$-algebras $(L,Q)$ and $(L',Q')$ are 
homotopy equivalent if and only if there exists an $L_\infty$-quasi-isomorphism 
between them.
\end{lemma}
\begin{proof}
Due to \cite[Prop.~2.8]{canonaco:1999a} every $L_\infty$-algebra $L$ is 
isomorphic to the product of a linear contractible one and a minimal 
one $(L,Q)\cong (V\oplus W,\widetilde{Q})$. 
This means $L\cong V\oplus W$ as vector spaces, such that 
$V$ is an acyclic cochain complex with differential $\D_V$ and $W$ 
is an $L_\infty$-algebra with codifferential $Q_W$ with 
$Q_{W,1}^1=0$. The codifferential $\widetilde{Q}$ on 
$\cc{\Sym}((V\oplus W)[1])$ is given on 
$v_1\vee \dots \vee v_m$with $v_1,\dots, v_k\in V$ and 
$v_{k+1},\dots, v_m\in W$ by 
	\begin{align*}
	\widetilde{Q}^1(v_1\vee\dots \vee v_m)=
	\begin{cases}
	-\D_V(v_1), & \text{ for } k=m=1\\
	Q_W^1(v_1\vee\dots\vee v_m), & \text{ for } k=0\\
	0, & \text{ else. } 
	\end{cases}
	\end{align*}
This implies in particular that the canonical maps
	\begin{align*}
	I_W\colon W\longrightarrow V \oplus W \ \text{ and } \ 
	P_W\colon V\oplus W\longrightarrow W 
	\end{align*}
are $L_\infty$-morphisms. We want to show now that $I_W\circ P_W\sim \id$. 
Choose a contracting homotopy 
$h_V\colon V\to V[-1]$ with $h_V\D_V+\D_V h_V=\id_V$ and define the maps
	\begin{align*}
	P(t)\colon V\oplus W\ni (v,w)\longmapsto (tv,w)\in V\oplus W
	\end{align*}	 
and 
	\begin{align*}
	H(t)\colon V\oplus W\ni (v,w)\longmapsto (-h_V(v),0)\in V\oplus W.
	\end{align*}
Note that $P(t)$ is a path of $L_\infty$-morphisms by the 
explicit form of the codifferential. We clearly have
\begin{align*}
	\frac{\D}{\D t} P^1_1(t) 
	= 
	\pr_V
	=
	\widetilde{Q}^1_1 \circ H(t) + H(t)\circ \widetilde{Q}^1_1
	=
	\widehat{Q}^1_1(H(t)) 
\end{align*}
since $h_V$ is a contracting homotopy. This implies
\begin{equation*}
  \frac{\D}{\D t} P(t)
	=
	\widehat{Q}^1(H(t)\vee \exp(P(t))
\end{equation*} 
since $\image (H(t))\subseteq V$ and as the higher brackets of 
$\widetilde{Q}$ vanish on $V$. Since 
$P(0)=I_W\circ P_W$ and $P(1)=\id$ we conclude that 
$I_W\circ P_W\sim \id$.
We choose a similar splitting for a $L'=V'\oplus W'$
 with the same
properties and consider an $L_\infty$-quasi-isomorphism 
$F\colon L\to L'$. Since $I_W, I_{W'},P_W$ and 
$P_{W'}$ are $L_\infty$-quasi-isomorphisms and we have that 
	\begin{align*}
	F_W= P_{W'}\circ F\circ I_W\colon W\longrightarrow W'
	\end{align*}	 
an $L_\infty$-isomorphism. Hence it invertible and we denote the inverse 
$G_{W'}$. We define now 
	\begin{align*}
	G=I_W\circ G_{W'}\circ P_{W'}\colon L'\longrightarrow L.
	\end{align*}	 
Since by Proposition~\ref{prop:CompofHomotopicHomotopic} and 
Proposition~\ref{prop:preCompofHomotopicHomotopic} compositions of homotopic 
$L_\infty$-morphisms with an $L_\infty$-morphism are again homotopic, we get
	\begin{align*}
	F\circ G&
	= F\circ I_W\circ G_{W'}\circ P_{W'}
	\sim I_{W'}\circ P_{W'}\circ  F\circ I_W\circ G_{W'}\circ P_{W'}\\&
	=I_{W'}\circ F_W\circ G_{W'}\circ P_{W'}=I_{W'}\circ P_{W'}\sim \id
	\end{align*}
and similarly $G\circ F \sim \id$.   

The other direction follows from Proposition~\ref{prop:PropertiesofHomotopicMorphisms}. Suppose $F\circ G \sim \id$ and 
$G\circ F \sim \id$, then we know that $F_1^1\circ G_1^1$ and 
$G_1^1\circ F_1^1$ are both chain homotopic to the identity. Therefore, 
$F$ and $G$ are $L_\infty$-quasi-isomorphisms.  
 \end{proof}
 
\begin{corollary}
Let $F\colon (L,Q)\to(L',Q')$ be a an 
$L_\infty$-quasi-isomorphism with two given quasi-inverses
$G,G'\colon (L',Q')\to (L,Q)$ in the sense of 
Definition~\ref{def:HomEquLinftyAlgs}. Then one has $G\sim G'$. 
\end{corollary} 
\begin{proof}
One has $G\sim G\circ (F\circ G')=(G\circ F)\circ G'\sim G'$.
\end{proof}

\subsection{Homotopy Equivalence between Twisted Morphisms}

Let now $F \colon (\liealg{g},Q) \rightarrow (\liealg{g}',Q')$ be an 
$L_\infty$-morphism between DGLAs with complete descending and 
exhaustive filtrations.   
Instead of comparing the twisted morphisms $F^\pi$ and 
$F^{\pi'}$ with respect to two equivalent Maurer-Cartan elements 
$\pi$ and $\pi'$, we consider 
for simplicity just a Maurer-Cartan element $\pi \in \mathcal{F}^1\liealg{g}^1$ 
 equivalent to zero via 
$\pi = \exp([g,\argument])\acts 0$, i.e. $\lambda(t)=g = \dot{A}(t)\in 
\mathcal{F}^1\widehat{\liealg{g}^0[t]}$. 
Then we know that $0$ and $S = F^1(\cc{\exp}(\pi))\in \mathcal{F}^1(\liealg{g}')^1$
are equivalend 
Maurer-Cartan elements in $(\liealg{g}',\D')$. Let the equivalence 
be implemented by an $A'(t)\in\mathcal{F}^1\widehat{(\liealg{g}')^0[t]}$ as in Proposition~\ref{prop:HomEquvsGaugeEqu}. 
Then we have the diagram
\begin{equation}
  \label{eq:TwistingofMorph}
	\begin{tikzcd}
	& (\liealg{g}',\D') \arrow[rd, "e^{[A'(1),\argument]}", bend left=12] \arrow[dd, Rightarrow,  shorten >=12pt,  shorten <=12pt] & \\
    (\liealg{g},\D) 	\arrow[ur,"F",bend left=12] \arrow[dr,swap, "e^{[A(1),\argument]}",bend right=12] & & 
    (\liealg{g}',\D' + [S,\argument])\\
	&(\liealg{g},\D + [\pi,\argument]) \arrow[ur, swap,"F^\pi", bend right=12] &
	\end{tikzcd}
\end{equation}
where $e^{[A(1),\argument]}$ and $e^{[A'(1),\argument]}$ are well-defined by 
the completeness of the filtrations. In the following we show that it commutes 
up to homotopy, which is indicated by the vertical arrow. 

\begin{proposition}
  \label{prop:TwistMorphHomEqu}
  The $L_\infty$-morphisms $F$ and $e^{[-A'(1),\argument]}\circ F^{\pi} \circ 
	e^{[A(1),\argument]}$ are homotopic, i.e. 
	gauge equivalent Maurer-Cartan elements 
	in $\Hom(\cc{\Sym}(\liealg{g}[1]),\liealg{g}')$.
\end{proposition}
The candidate for the path between $F$ and 
$e^{[-A'(1),\argument]}\circ F^{\pi} \circ 
	e^{[A(1),\argument]}$ is 
\begin{equation*}
  F(t) 
	= 
	e^{[-A'(t),\argument]}\circ F^{\pi(t)} \circ 
	e^{[A(t),\argument]}.
\end{equation*}
However, $F(t)$ is not necessarily in the completion 
$\widehat{(\Hom(\cc{\Sym}(\liealg{g}[1]),\liealg{g}')^1[t])}$ 
with respect to the 
filtration from \eqref{eq:FiltrationConvLieAlg} since for example
\begin{align*}
  F(t) \mod \mathcal{F}^2\Hom(\cc{\Sym}(\liealg{g}[1]),\liealg{g}')[[t]]
	=
	 e^{[-A'(t),\argument]}\circ F^{\pi(t)}_1 \circ 
	e^{[A(t),\argument]}
\end{align*}
is in general no polynomial in $t$. To solve this problem we 
introduce a new filtration on the convolution DGLA 
$\liealg{h} = \Hom(\cc{\Sym}(\liealg{g}[1]),\liealg{g}')$ 
that takes into 
account the filtrations on $\cc{\Sym}(\liealg{g}[1])$ and $\liealg{g}'$:
\begin{align}
  \label{eq:FiltrationConvLieAlg2}
  \begin{split}
  \liealg{h}
	=
	& \mathfrak{F}^1\liealg{h}
	\supset
	\mathfrak{F}^2\liealg{h}
	\supset \cdots \supset
	\mathfrak{F}^k\liealg{h}
	\supset \cdots \\
	\mathfrak{F}^k\liealg{h}
	& =
	\sum_{n+m=k}
	\left\{ f \in \Hom(\cc{\Sym}(\liealg{g}[1]),\liealg{g}') \mid 
	f \at{\Sym^{<n}(\liealg{g}[1])}=0\quad \text{ and }\quad 
	f\colon\mathcal{F}^\bullet \rightarrow 
	\mathcal{F}^{\bullet +m}\right\}.
	\end{split}
\end{align}
Here the filtration on $\cc{\Sym}(\liealg{g}[1])$ is the product filtration 
induced by
\begin{equation*}
  \mathcal{F}^k(\liealg{g}[1] \otimes \liealg{g}[1])
	=
	\sum_{n+m=k} \image\left(  \mathcal{F}^n\liealg{g}[1] \otimes 
	\mathcal{F}^m\liealg{g}[1] \rightarrow \liealg{g}[1]\otimes \liealg{g}[1]\right),
\end{equation*}
see e.g. \cite[Section~1]{dotsenko.shadrin.vallette:2018}.

\begin{proposition}
  The above filtration \eqref{eq:FiltrationConvLieAlg2} is a complete descending 
	filtration on the convolution DGLA $\Hom(\cc{\Sym}(\liealg{g}[1]),\liealg{g}')$.
\end{proposition}
\begin{proof}
The filtration is obviously descending and 
$\liealg{h}=\mathfrak{F}^1\liealg{h}$ since we consider in the convolution DGLA only maps that are compatible with respect to the filtration. It is compatible 
with the convolution DGLA structure and complete since $\liealg{g}'$ is complete.
\end{proof}

Thus we can finally prove Proposition~\ref{prop:TwistMorphHomEqu}.
\begin{proof}[of Prop.~\ref{prop:TwistMorphHomEqu}]
  The path $F(t) = e^{[-A'(t),\argument]}\circ F^{\pi(t)} \circ 
	e^{[A(t),\argument]}$ is an element in the completion
	$\widehat{(\Hom(\cc{\Sym}(\liealg{g}[1]),\liealg{g}')[1])^{0}[t]}$ 
	with respect to the filtration from \eqref{eq:FiltrationConvLieAlg2}.
	This is clear since $A(t)\in \mathcal{F}^1 \widehat{\liealg{g}^0[t]}$, 
	$A'(t)\in \mathcal{F}^1 \widehat{(\liealg{g}')^0[t]}$ and 
	$\pi(t)\in \mathcal{F}^1 \widehat{\liealg{g}^1[t]}$ imply that
	\begin{align*}
	  \sum_{i=1}^{n-1}
	  e^{[-A'(t),\argument]}\circ F^{\pi(t)}_i \circ 
	  e^{[A(t),\argument]}
	  \mod \mathfrak{F}^n(\Hom(\cc{\Sym}(\liealg{g}[1]),\liealg{g}')[1])[[t]]
	\end{align*}
	is polynomial in $t$. Moreover, $F(t)$ satisfies by \eqref{eq:ODEforA}
	\begin{align*}
	  \frac{\D F(t)}{\D t}
		& = 
		-\exp([-A'(t),\argument]) \circ 
		\left[\lambda'(t) , 
	  \argument \right]
	  \circ F^{\pi(t)} \circ 
		e^{[A(t),\argument]}  
		+ 
		e^{[-A'(t),\argument]}\circ F^{\pi(t)} \circ 
		\left[\lambda(t), 
	  \argument \right] \circ 
	  e^{[A(t),\argument]} \\
		& \quad 
		+ e^{[-A'(t),\argument]}\circ \frac{\D F^{\pi(t)}}{\D t} \circ 
	  e^{[A(t),\argument]} .
	\end{align*}
	But we have
	\begin{align*}
	  \frac{\D F^{\pi(t)}_k}{\D t}&(X_1 \vee \cdots \vee X_k)
		=
		F_{k+1}^{\pi(t)}(Q^{\pi(t),1}_1(\lambda(t)) \vee X_1 \vee \cdots 
		\vee X_k) \\
		& =
		F_{k+1}^{\pi(t)}(Q^{\pi(t),k+1}_{k+1}(\lambda(t)\vee X_1\vee \cdots\vee X_k))
		+
		F_{k+1}^{\pi(t)}(\lambda(t) \vee Q^{\pi(t),k}_k(X_1 \vee \cdots \vee X_k)))\\
		& =
		Q^{S(t),1}_1 F_{k+1}^{\pi(t),1}(\lambda(t)\vee X_1 \vee \cdots \vee X_k) +
		Q^{S(t),1}_2 F_{k+1}^{\pi(t),2}(\lambda(t)\vee X_1 \vee \cdots \vee X_k) \\
		& \quad 
		-  F_{k}^{\pi(t),1}\circ Q^{\pi(t),k}_{k+1}
		(\lambda(t)\vee X_1 \vee \cdots \vee X_k) 
		+
		F_{k+1}^{\pi(t)}(\lambda(t) \vee Q^{\pi(t),k}_k(X_1 \vee \cdots \vee X_k)).
	\end{align*}
	Setting now $\lambda_k^F(t)(\cdots) 
	= F_{k+1}^{\pi(t)}( \lambda(t)\vee \cdots)$ we get
	\begin{align*}
	  \frac{\D F^{\pi(t)}_k}{\D t}
		=
		\widehat{Q}^{t,1}_1 (\lambda_k^F(t)) + 
		\widehat{Q}^{t,1}_2(\lambda^F(t)\vee F^{\pi(t)}) - 
		F^{\pi(t)}_k\circ[\lambda(t),\argument] +[\lambda'(t),\argument] \circ 
		F^{\pi(t)}_k.
	\end{align*}
	Thus we get
	\begin{align*}
	  \frac{\D F(t)}{\D t}
		& =
		e^{[-A'(t),\argument]}\circ \left(\widehat{Q}^{t,1}_1 (\lambda^F(t)) + 
		\widehat{Q}^{t,1}_2(\lambda^F(t)\vee F^{\pi(t)})\right) \circ 
	  e^{[A(t),\argument]} \\
		& =
		\widehat{Q}^{1}_1 (e^{[-A'(t),\argument]}\lambda^F(t)e^{[A(t),\argument]}) + 
		\widehat{Q}^{1}_2(e^{[-A'(t),\argument]}\lambda^F(t)e^{[A(t),\argument]}\vee 
		F(t))		
	\end{align*}
	since $\exp([A(t),\argument])$ and 
	$\exp([A'(t),\argument])$ commute with the brackets and intertwine the 
	differentials. Thus $F(0) = F$ and $F(1)$ are homotopy equivalent and by 
	Theorem~\ref{thm:DGALHomvsGaugeEquiv} also gauge equivalent.
\end{proof}

\section{Homotopic Globalizations of the Kontsevich Formality}
\label{sec:ApplicationtoGlobalization}

Now we want to apply the above general results to the globalization of 
the Kontsevich formality to smooth real manifolds $M$. More precisely, 
the globalization procedure proved by Dolgushev in  
\cite{dolgushev:2005a,dolgushev:2006a} depends on the choice 
of a torsion-free covariant derivative on $M$ and we show that the globalizations 
with respect to two different covariant derivatives are homotopic.

\subsection{Preliminaries: Globalization Procedure}

Let us briefly recall the globalization procedure from 
from \cite{dolgushev:2005a,dolgushev:2006a} to establish the notation.
  \begin{itemize}
	  \item $\mathcal{T}_{poly}^k$ denotes the bundle of 
		      \emph{formal fiberwise polyvector fields} 
	        of degree $k$ over $M$. Its sections are $\Cinfty(M)$-linear 
				  operators $v\colon \Anti^{k+1} \Secinfty(SM)\rightarrow \Secinfty(SM)$ 
				  of the local form
          \begin{equation*}
            v 
	          =
	          \sum_{p=0}^\infty v_{i_1\dots i_p}^{j_0 \dots j_k}(x) 
	          y^{i_1}\cdots y^{i_p} \frac{\del}{\del y^{j_0}} \wedge \dots 
	          \wedge \frac{\del}{\del y^{j_k}}.
          \end{equation*}
		\item Analogously, the sections of 
		      \emph{formal fiberwise differential operators} 
					$\mathcal{D}_{poly}^k$ are $\Cinfty(M)$-linear operators 
					$X\colon \bigotimes^{k+1} \Secinfty(SM)\rightarrow \Secinfty(SM)$ 
					of the local form
	        \begin{equation*}
	          X
		        =
		        \sum_{\alpha_0,\dots,\alpha_k} 
		        \sum_{p=0}^\infty X_{i_1\dots i_p}^{\alpha_0 \dots \alpha_k}(x) 
		        y^{i_1}\cdots y^{i_p} \frac{\del}{\del y^{\alpha_0}} \otimes \dots 
		        \otimes \frac{\del}{\del y^{\alpha_k}}.		
	        \end{equation*}
	        Here $X_{i_1\dots i_p}^{\alpha_0 \dots \alpha_k}$ are symmetric in 
					the indices $i_1,\dots,i_p$ and $\alpha$ are multi-indices 
					$\alpha = (j_0,\dots,j_k)$. Moreover, the sum in the orders of 
					the derivatives is finite.
    \item $D = - \delta + \nabla + [A,\argument] = \D + [B,\argument]$ is the Fedosov 
	        differential, where 
	        $\delta = [\D x^i \frac{\del}{\del y^i},\argument]$, $\nabla = 
				  \D x^i \frac{\del}{\del x^i} - [\D x^i \Gamma^k_{ij}(x) y^j
				  \frac{\del}{y^k},\argument]$ 
				  with Christoffel symbols $\Gamma^k_{ij}$ of 
				  a torsion-free connection on $M$ with curvature $R = 
          -\frac{1}{2} \D x^i \D x^j (R_{ij})^k_l(x) y^l 
				  \frac{\del}{\del y^k}$, and 
				  $A\in \Omega^1(M,\mathcal{T}_{poly}^0)
				  \subseteq \Omega^1(M,\mathcal{D}_{poly}^0)$
				  is the unique solution of 
				  \begin{align*}
				  \begin{cases}
				  \delta( A) & =R+\nabla A+\frac{1}{2}[A,A],\\
				  \delta^{-1}(A) & =r, \\
					\sigma(A) & = 0.
				  \end{cases}
				  \end{align*}
		    Here $r\in\Omega^0(M,\mathcal{T}_\poly^0)$ is arbitrary but fixed 
		    and has vanishing constant and linear term with respect to the 
				$y$-variables. We
		     refer to $(\nabla,r)$  as globalization data. 
	  \item $\tau \colon \Secinfty_\delta(\mathcal{T}_\poly) \rightarrow 
	        Z^0(\Omega(M,\mathcal{T}_\poly),D)$ denotes the Fedosov Taylor series, given by 
				  \begin{equation*}
				    \tau(a)
					  =
					  a + \delta^{-1}(\nabla\tau(a) + [A,\tau(a)]).
				  \end{equation*}
				  Here one has $\Secinfty_\delta(\mathcal{T}_\poly) = \{ v 	=
				  \sum_k v^{j_0\dots j_k}(x) \frac{\del}{\del y^{j_0}} \wedge 
				  \cdots \wedge \frac{\del}{\del y^{j_k}}\}$, analogously for the 
					polydifferential operators. In addition, $\del_M$ denotes the fiberwise 
					Hochschild differential.
	  \item $\nu \colon \Secinfty_\delta(\Tpolyres) \rightarrow \Tpoly(M)$ 
	        is given by
				  \begin{equation*}
            \nu ( w)(f_0,\dots,f_k) 
            =
            \sigma w(\tau(f_0),\dots,\tau(f_k))
						\quad \text{ for } \quad
						f_1,\dots,f_k \in \Cinfty(M),
           \end{equation*}
				  where $\sigma$ sets the $\D x^i$ and $y^j$ coordinates to zero, 
					analogously for the polydifferential operators.
	  \item $\mathcal{U}^B$ is the fiberwise formality of Kontsevich $\mathcal{U}$
		      twisted by 
	        \begin{equation*}					
				    B 
					  = 
					  D - \D 
            =
            -\D x^i \frac{\del }{\del y^i} - 
					  \D x^i \Gamma^k_{ij}(x)y^j\frac{\del}{\del y^k}
            + \sum_{p\geq 1} \D x^i A^k_{ij_1\dots j_p}(x)y^{j_1}\cdots y^{j_p}
            \frac{\del}{\del y^k}.
					  \end{equation*}
						By the properties of the Kontsevich formality the first two summands 
						do not contribute, i.e. $\mathcal{U}^B = \mathcal{U}^A$.
  \end{itemize}
	One obtains the diagram
  \begin{align*}
    T_{\mathrm{poly}}(M)
	  \stackrel{\tau \circ \nu^{-1}}{\longrightarrow}
	  (\Omega(M,\mathcal{T}_\poly),D)
	  \stackrel{\mathcal{U}^B}{\longrightarrow}
	  (\Omega(M,\mathcal{D}_\poly),D+\del_M)
	  \stackrel{\tau\circ \nu^{-1}}{\longleftarrow}
	  D_{\mathrm{poly}}(M),
  \end{align*}
  where $\tau \circ \nu^{-1}$ are quasi-isomorphisms of 
	DGLAs and where $\mathcal{U}^B$ is an $L_\infty$-quasi-isomorphism. 
	In a next step, the morphism 
	$\mathcal{U}^B \circ \tau \circ \nu^{-1}$ is modified 
	to a quasi-isomorphism
  \begin{equation*}
    U \colon
    \Tpoly(M)
	  \longrightarrow 
	  (Z^0_D(\Omega(M,\Dpolyres)),\del_M,[\,\cdot\,,\,\cdot\,]_G) ,		
  \end{equation*}
	see \cite[Prop.~5]{dolgushev:2005a}.
	By \cite[Lemma~1]{dolgushev:2007a} we know that 
	$\mathcal{U}^B \circ \tau \circ \nu^{-1}$ 
	and $U$ are homotopic.
The desired quasi-isomorphism $\boldsymbol{U}^{(\nabla,r)} = 
\nu \circ \sigma \circ U\colon \Tpoly(M) \rightarrow 
\Dpoly(M)$ is then the composition of $U$ with 
the DGLA isomorphism
\begin{equation*}
  \nu \circ \sigma \colon 
	(Z^0_D(\Omega(M,\Dpolyres)),\del_M,[\,\cdot\,,\,\cdot\,]_G)
	\longrightarrow
	(\Dpoly(M),\del,[\argument,\argument]_G).
\end{equation*}

\begin{corollary}
The formality $\boldsymbol{U}^{(\nabla,r)}$ induces a one-to-one correspondence 
between equivalent formal Poisson structures on $M$ and equivalent 
differential star products on $\Cinfty(M)[[\hbar]]$, i.e. a 
bijection
\begin{equation}
  \label{eq:ClassificationofStarProd}
	\boldsymbol{U}^{(\nabla,r)} \colon
	\Def(\Tpoly(M)[[\hbar]]) 
	\longrightarrow
	\Def(\Dpoly(M)[[\hbar]]).
\end{equation} 
\end{corollary}

\subsection{Explicit Construction of the Projection $L_\infty$-morphism}
\label{sec:ExplicitProjectionDolg}

As an alternative to the modification of the formality in 
\cite[Prop.~5]{dolgushev:2005a} we want to construct the 
$L_\infty$-quasi-inverse of $\tau\circ \nu^{-1}$. We want to 
use the construction from \cite[Prop.~3.2]{esposito.kraft.schnitzer:2020a:pre} 
that gives a formula for the $L_\infty$-quasi-inverse of an inclusion of 
DGLAs, see also \cite{loday.vallette:2012a} for the existence in more general 
cases. In our setting we have the contraction
\begin{equation}
  \label{eq:Contraction}
  \begin{tikzcd} 
	(\Dpoly(M),\del)
	\arrow[rr, shift left, "\tau \circ \nu^{-1}"] 
  &&   
  \left(\Omega(M,\Dpolyres), \del_M + D \right)
  \arrow[ll, shift left, "\nu \circ \sigma"]
	\arrow[loop right, 
	distance=3em, start anchor={[yshift=1ex]east}, end anchor={[yshift=-1ex]east}]{}{h} ,
\end{tikzcd} 
\end{equation}
where the homotopy $h$ with respect to $\del_M + D$ is 
constructed as follows:
As in the Fedosov construction in the symplectic setting one has a 
homotopy $D^{-1}$ for the differential $D$, see also 
\cite[Thm.~3]{dolgushev:2005a}:

\begin{proposition}
  The map 
	\begin{equation}
	 	\label{eq:HomotopyforFedosovDiff}
		D^{-1}
		=
		- \delta^{-1}\frac{1}{\id - [\delta^{-1}, \nabla + [A,\argument]]} 
		=
		- \frac{1}{\id - [\delta^{-1}, \nabla + [A,\argument]]} \delta^{-1}
	\end{equation}
	is a homotopy for $D$ on $\Omega(M,\Dpolyres)$, i.e. one has
	\begin{equation}
	  \label{eq:HomotopyEqD}
	  X 
		=
		D D^{-1} X + D^{-1}DX + \tau \sigma(X).
	\end{equation}
\end{proposition}
\begin{proof}
The proof is the same as in the symplectic setting, see e.g. 
\cite[Prop.~6.4.17]{waldmann:2007a}.
\end{proof}

If this homotopy is also compatible with the Hochschild differential $\del_M$, 
then we can indeed apply \cite[Prop.~3.2]{esposito.kraft.schnitzer:2020a:pre} 
to describe the $L_\infty$-morphism extending $\nu \circ \sigma$.
%
%
Let us denote by $(D^{-1})_{k+1}$ the extended homotopy on 
$\Sym^{k+1} (\Omega(M,\Dpolyres)[1])$ and let us write $Q_{\Dpolyres},
Q_{\Dpoly}$ for the induced codifferentials on the symmetric algebras. 
Then we get:

\begin{proposition}
  \label{prop:InftyProjectionDolg}
  The homotopy $D^{-1}$ anticommutes with $\del_M$, whence it is also a 
	homotopy for 
	$\del_M + D$. Therefore, one obtains an $L_\infty$-quasi-isomorphism 
	$P \colon \Sym (\Omega(M,\Dpolyres)[1]) \rightarrow 
	\Sym (\Dpoly(M)[1])$ with recursively defined structure maps
	\begin{equation}
	  \label{eq:defPinfty}
	  P_1^1 
		= 
		\nu \circ \sigma
		\quad \quad \text{ and } \quad \quad 
		P_{k+1}^1 
		= 
		(Q^1_{\Dpoly,2} \circ P^2_{k+1} - P^1_k \circ Q^k_{\Dpolyres,k+1}) 
		\circ (D^{-1})_{k+1}.
	\end{equation}
\end{proposition}
\begin{proof}
The fact that $D^{-1}$ anticommutes with $\del_M$ is clear as $\nabla + [A,\argument]$ and $\delta^{-1}$ anticommute with $\del$, and the rest 
follows directly from 
\cite[Prop.~3.2]{esposito.kraft.schnitzer:2020a:pre}.
\end{proof}

Summarizing, we obtain another global formality:

\begin{corollary}
Given the globalization data $(\nabla,r)$ there exists 
an $L_\infty$-quasi-isomorphism 
  \begin{equation}
  \label{eq:FormalityonM}
	  F^{(\nabla,r)}
  	=
	  P \circ \mathcal{U}^B \circ\tau \circ \nu^{-1} \colon 
	  \Tpoly(M)
	  \longrightarrow
	  \Dpoly(M)
  \end{equation}
  with $F^1_1$ being the Hochschild-Kostant-Rosenberg map.
\end{corollary}
\begin{proof}
We immediately get $F^{(\nabla,r),1}_1 = P^1_1 \circ(\mathcal{U}^B)^1_1 \circ 
\tau\circ\nu^{-1}
= \nu\circ \sigma\circ \mathcal{U}^1_1 \circ \tau \circ \nu^{-1}$ and the 
statement follows since $\mathcal{U}^1_1$ is the fiberwise 
Hochschild-Kostant-Rosenberg map.
\end{proof}

The higher structure maps of $P^1_{k+1}$ of the 
$L_\infty$-projection contain copies of the homotopy $D^{-1}$ that decrease 
the antisymmetric form degree. Therefore, they vanish on 
$\Omega^0(M,\Dpolyres)$ and are needed to 
get rid of the form degrees arising from the twisting with $B$,  
analogously to the modifying of the formality from $\mathcal{U}^B\circ\tau\circ 
\nu^{-1}$ to $U$.

As a last point, we want to remark that we can use the 
$L_\infty$-projection $P$ to obtain a splitting of 
$\Omega(M,\Dpolyres)$ similar to the one used in the 
proof of Lemma~\ref{lemma:HomEquvsQuasiIso}: 
Instead of splitting into the product of 
the cohomology as minimal $L_\infty$-algebra and a linear contractible 
one, we can prove in our setting:

\begin{lemma}
\label{lemma:Dpolysplits}
  One has an $L_\infty$-isomorphism 
	\begin{equation}
	  L \colon 
		\Omega(M,\Dpolyres)
		\longrightarrow 
		\Dpoly(M)\oplus \image [D,D^{-1}].
	\end{equation}
	Here the $L_\infty$-structure on $\Dpoly(M)$ is the usual one consisting 
	of Gerstenhaber bracket and Hochschild differential $\del$ and on 
	$\image [D,D^{-1}]$ the $L_\infty$-structure is just given by the 
	differential $\del_M+ D$. The $L_\infty$-structure on 
	$\Dpoly(M)\oplus \image [D,D^{-1}]$ is the product $L_\infty$-structure.
\end{lemma}
\begin{proof}
By Proposition~\ref{prop:InftyProjectionDolg} we already have an 
$L_\infty$-morphism $P \colon \Omega(M,\Dpolyres)\rightarrow \Dpoly(M)$ 
with first structure map $\nu \circ \sigma$. Now we construct an 
$L_\infty$-morphism $F \colon \Omega(M,\Dpolyres)\rightarrow 
\image [D,D^{-1}]$. We set $F^1_1 = DD^{-1}+D^{-1}D$ and  
$F^1_n = - D^{-1}\circ F^1_{n-1} \circ (Q_{\Dpolyres})^{n-1}_n$ for $n\geq 2$ and note 
that in particular $F^1_k = 0$ for $k\geq 3$ by $D^{-1}D^{-1}=0$. 
In the following, we denote by $Q_{\Dpolyres}$ the $L_\infty$-structure on 
$\Omega(M,\Dpolyres)$ and 
by $\widetilde{Q}$ the $L_\infty$-structure on $\image [D,D^{-1}]$ with 
$\widetilde{Q}^1_1 = -(\del_M +D)$ as only vanishing component. 
We have $F^1_n = D^{-1} \circ L_{\infty,n}$ with $ L_{\infty,n}	=
- F^1_{n-1} \circ (Q_{\Dpolyres})^{n-1}_n$.
By \cite[Lemma~3.1]{esposito.kraft.schnitzer:2020a:pre} we know that 
if $F$ is an $L_\infty$-morphism up to order $n$, i.e. if 
$(\widetilde{Q}F)_k^1= (FQ_{\Dpolyres})^1_k$ for all $k\leq n$, then one has 
$\widetilde{Q}^1_1\circ  L_{\infty,n+1}= - L_{\infty,n+1}\circ 
(Q_{\Dpolyres})^{n+1}_{n+1}$.
By Proposition~\ref{prop:InftyProjectionDolg} 
we know that $F$ is an $L_\infty$-morphism up to order one. Assuming 
it defines an $L_\infty$-morphism up to order $n$, then we get 
with \eqref{eq:HomotopyEqD}
\begin{align*}
  \widetilde{Q}^1_1 \circ F^1_{n+1}
	& =
	- (\del_M + D)\circ D^{-1} \circ L_{\infty,n+1} \\
	& =
	D^{-1} \circ \del_M \circ L_{\infty,n+1}
	- L_{\infty,n+1} + D^{-1}\circ D \circ L_{\infty,n+1} + 
	\tau\circ \sigma \circ L_{\infty,n+1} \\
	& =
	- L_{\infty,n+1} - D^{-1}\circ \widetilde{Q}^1_1 \circ L_{\infty,n+1} \\
	& =
	- L_{\infty,n+1} + F^1_{n+1}\circ (Q_{\Dpolyres})^{n+1}_{n+1}.
\end{align*}
Thus $F$ is an $L_\infty$-morphism up to order $n+1$ and therefore an 
$L_\infty$-morphism.

The universal property of the product gives the desired 
$L_\infty$-morphism $L= P\oplus F$ which is even an isomorphism since its first 
structure map $(\nu\circ \sigma) \oplus (DD^{-1}+D^{-1}D)$ is an isomorphism 
with inverse $(\tau\circ \nu^{-1}) \oplus \id$, see e.g. 
\cite[Prop.~2.2]{canonaco:1999a}. 
\end{proof}

\subsection{Homotopic Global Formalities}

The above globalization of the Kontsevich Formality depends on the choice of 
a covariant derivative. We want to show that globalizations with respect to 
different covariant derivatives are homotopic in the sense of 
Definition~\ref{def:homotopicMorph}. The ideas are similar to those in 
the proof of \cite[Theorem~2.6]{bursztyn.dolgushev.waldmann:2012a}: 
observe that changing the covariant derivative corresponds to twisting with a 
Maurer-Cartan element which is equivalent to zero, and apply 
Proposition~\ref{prop:TwistMorphHomEqu}.

\begin{remark}[Filtrations on Fedosov resolutions] 
  In order to apply Proposition~\ref{prop:TwistMorphHomEqu} we need 
	complete descending and exhaustive filtrations on the Fedosov 
	resolutions. As in \cite[Appendix~C]{bursztyn.dolgushev.waldmann:2012a} 
	we  assign to $\D x^i$ and $y^i$ the degree $1$ and to $\frac{\del}{\del y^i}$ 
	the degree $-1$ and consider the induced descending filtration. The filtration
	on $\Omega(M,\Tpolyres)$ is complete and bounded below since
	\begin{equation*}
	  \Omega(M,\Tpolyres)
		\cong 
		\varprojlim \Omega(M,\Tpolyres) / \mathcal{F}^k\Omega(M,\Tpolyres)
		\quad \quad \text{ and } \quad \quad
	  \Omega(M,\Tpolyres) 
		=
		\mathcal{F}^{-d}\Omega(M,\Tpolyres),
	\end{equation*}
	where $d$ is the dimension of $M$. In the case of the differential operators 
	the filtration is unbounded in both directions. 
	Instead of $\Dpolyres$ we consider from now on its $y$-adic completion 
	without changing the notation. This is 
	the completion with respect to the filtration induced by assigning 
	$y^i$ the degree $1$ and $\frac{\del}{\del y^i}$ the degree $-1$. 
	The globalization of the formality works just
	the same and one obtains the desired properties
	\begin{equation*}
	  \Omega(M,\Dpolyres)
		\cong 
		\varprojlim \Omega(M,\Dpolyres) / \mathcal{F}^k\Omega(M,\Dpolyres)
		\quad \quad \text{ and } \quad \quad
	  \Omega(M,\Dpolyres) 
		=
		\bigcup_k \mathcal{F}^{k}\Omega(M,\Dpolyres).
	\end{equation*}
\end{remark}

Let $(\nabla',r')$ be a second pair of globalization data, then  there is a second diagram 
\begin{align*}
  T_{\mathrm{poly}}(M)
  \stackrel{\tau' \circ \nu^{-1}}{\longrightarrow}
  (\Omega(M,\Tpolyres),D')
  \stackrel{\mathcal{U}^{B'}}{\longrightarrow}
  (\Omega(M,\Dpolyres),D'+\del_M)
  \stackrel{\tau'\circ (\nu')^{-1}}{\longleftarrow}
  D_{\mathrm{poly}}(M),
\end{align*}
where 
\begin{equation*}
  D' 
	= 
	-\delta + \nabla' + [A',\argument] 
	= 
	\D + [B',\argument],
\end{equation*}
and $A'$ is the unique solution of 
	\begin{align*}
				  \begin{cases}
				  \delta (A') & =R'+\nabla' A'+\frac{1}{2}[A',A'],\\
				  \delta^{-1}(A')& =r', \\
					\sigma(A') & = 0.
				  \end{cases}
	\end{align*}
In the case of polyvector fields one easily sees that $\nu=\nu'$. 
Note
\begin{equation*}
  \nabla' - \nabla
	=
	\left[-\D x^i y^j (\Gamma'^k_{ij}-\Gamma^k_{ij})\frac{\del}{\del y^k} ,\argument\right]
	=
	[\delta s,\argument]
\end{equation*}
for $s = - \frac{1}{2}y^iy^j (\Gamma'^k_{ij}-\Gamma^k_{ij})\frac{\del}{\del y^k}$.
Thus we get
\begin{equation}
  \label{eq:Aprime}
	D' 
	= 
	-\delta + \nabla + [A' + \delta s, \argument]
	=
	-\delta + \nabla + [\widetilde{A},\argument]
\end{equation}
and since $R' = R + \nabla \delta s + \frac{1}{2}[\delta s,\delta s]$ we know 
that $\widetilde{A} = A' + \delta s$ is the unique solution of 
	\begin{align*}
				  \begin{cases}
				  \delta (\widetilde{A})& =R+\nabla \widetilde{A}
				  +\frac{1}{2}[\widetilde{A},\widetilde{A}],\\
				  \delta^{-1}(\widetilde{A})& =r'+s , \\
					\sigma(\widetilde{A}) & = 0.
				  \end{cases}
	\end{align*}


As in 
\cite[Appendix~C]{bursztyn.dolgushev.waldmann:2012a} one can now show 
that $B$ and $B'$ can be interpreted as equivalent Maurer-Cartan elements:

\begin{proposition}
  \label{prop:BBprimeEquivalent}
	There exists an element
	\begin{equation}
    \label{eq:RecDefofh}
	  h 
	  \in
	  \mathcal{F}^1\Omega^0(M,\Tpolyres^0)
  \end{equation}
	that is at least quadratic in the fiber coordinates $y$ such that 
	one has 
	\begin{equation}
	  B' - B
		=
		\widetilde{A} - A
		=
		-\frac{\exp([h,\argument])-\id}{[h,\argument]}Dh
		\in 
		\mathcal{F}^1\Omega^1(M,\Tpolyres^0)
	\end{equation}
	and 
  \begin{equation}
    \exp([h,\argument])\circ D \circ\exp([-h,\argument])
	  =
	  D'.
  \end{equation}
  Thus $B'-B$ is gauge equivalent to zero in	
	in $(\Omega(M,\Tpolyres),D)$ and 
	$(\Omega(M,\Dpolyres),D+ \del_M)$, 
	where $h$ implements the gauge equivalences.
\end{proposition}
\begin{proof}
For the existence of the element $h\in \mathcal{F}^1\Omega^0(M,\Tpolyres^0)$ 
encoding the gauge equivalence in the polyvector fields see 
\cite[Appendix~C]{bursztyn.dolgushev.waldmann:2012a}. Thus we 
have a path
\begin{equation*}
  B(t)
	=
	-\frac{\exp([th,\argument])-\id}{[th,\argument]}D(th)
	\in \mathcal{F}^1\widehat{\Omega^1(M,\Tpolyres^0)[t]}
\end{equation*} 
that satisfies $B(0)=0$, $B(1)=B'-B$ and
\begin{equation*}
	\frac{\D B(t)}{\D t}
	=
	Q^1(\lambda(t) \vee \exp(B(t)))
	\quad \quad \text{with}\quad\quad
	\lambda(t)=h.
\end{equation*}
The formality $\mathcal{U}^B$ satisfies in the notation of 
Proposition~\ref{prop:FmapsEquivMCtoEquiv}
\begin{equation*}
  \widetilde{B}(t)
	=
	\mathcal{U}^{B,1}(\cc{\exp}(B(t))
	=
	B(t)
	\quad \quad\text{ and }\quad \quad
	\widetilde{\lambda}(t)
	=
	\mathcal{U}^{B,1}(h\vee \exp(B(t)))
	=
	h
\end{equation*}
since the higher orders of the Kontsevich formality vanish if one only inserts 
vector fields.
\end{proof}

Now it follows directly from Proposition~\ref{prop:TwistMorphHomEqu} 
and $(\mathcal{U}^B)^{B'-B}  = \mathcal{U}^{B'}$
that the twisted formalities are homotopic.

\begin{corollary}
  The $L_\infty$-morphisms $\mathcal{U}^{B}$ and $e^{-[h,\argument]}\circ 
	\mathcal{U}^{B'}\circ e^{[h,\argument]}$ are homotopic.
\end{corollary}
	
Moreover, the Fedosov Taylor series is compatible in the following sense:

\begin{corollary}
  For all $X \in T_\poly(M)$ one has
	\begin{equation}
	  e^{[h,\argument]} \circ \tau \circ \nu^{-1}(X)
		=
		\tau' \circ (\nu')^{-1}(X).
	\end{equation}
\end{corollary}
\begin{proof}
By the above proposition $\exp([h,\argument])$ maps the kernel of $D$ into the 
kernel of $D'$. Therefore,
\begin{align*}
  e^{[h,\argument]} \circ \tau \circ \nu^{-1} (X)
	=
	\tau' \circ \sigma \circ  e^{[h,\argument]} \circ \tau  \circ \nu^{-1}(X)
	=
  \tau' \circ \nu^{-1} (X)
\end{align*}
since $h$ is at least quadratic in the $y$ coordinates.
\end{proof}

Similarly, one has on the differential operator side the following identity:

\begin{lemma}
  \label{lemma:hintertwinesProj}
  For all $X \in  Z^0_{D'}(\Omega(M,\Dpolyres))$ one has
	\begin{equation}
	  \nu\circ\sigma \circ e^{-[h,\argument]}(X)
		=
		\nu'\circ \sigma(X).
	\end{equation}
\end{lemma}
\begin{proof}
Using the definition of $\nu$ 		
we compute for $X\in  Z^0_{D'}(\Omega(M,\Dpolyres^{n-1}))$ 
and $f_1,\dots,f_n \in \Cinfty(M)$				
\begin{align*}
  ( \nu\circ\sigma \circ e^{-[h,\argument]} X) (f_1,\dots,f_n)
	& =	
	\sigma ((\sigma \circ e^{-[h,\argument]}X) 
	(\tau(f_1),\dots,\tau(f_n)) ) \\
	& =
	\sigma ( e^{-h}( X  
	(e^{h}\tau(f_1),\dots,e^{h}\tau(f_n))))  \\
	& =
	\sigma ((\sigma \circ X ) 
	(\tau'(f_1),\dots,\tau'(f_n)))  \\
	& =
	(\nu'\circ \sigma X)(f_1,\dots,f_n)
\end{align*}
and the statement is shown.
\end{proof}

As a last preparation we want to compare the two different 
$L_\infty$-projections 
$P'$ and $P \circ e^{-[h,\argument]}$ from $(\Omega(M,\Dpolyres),\del_M+D')$ 
to $(\Dpoly(M),\del)$. 

\begin{lemma}
  \label{lemma:ProjHomotopic}
  The $L_\infty$-projections $P'$ and $P \circ e^{-[h,\argument]}$ are 
	homotopic.
\end{lemma}
\begin{proof}
Since the higher structure maps of $P$ and $P'$ vanish on the zero forms, 
we have by Lemma~\ref{lemma:hintertwinesProj} 
\begin{align*}
  P\circ e^{-[h,\argument]} \circ \tau'\circ (\nu')^{-1} 
  =
  \nu \circ \sigma \circ e^{-[h,\argument]} \circ \tau' \circ (\nu')^{-1}
	=
	\nu'\circ \sigma \circ \tau' \circ( \nu')^{-1}
  = 
	\id_{\Dpoly(M)}.
\end{align*}
Instead of directly using Lemma~\ref{lemma:HomEquvsQuasiIso}, we 
recall the splitting from Lemma~\ref{lemma:Dpolysplits} and adapt the 
proof of Lemma~\ref{lemma:HomEquvsQuasiIso}. Define 
\begin{equation*}
  M(t)\colon
	\Dpoly(M)\oplus \image [D,D^{-1}] \ni 
	(D_1,D_2)
	\longmapsto 
	(D_1,tD_2) \in 
	\Dpoly(M)\oplus \image [D,D^{-1}]
\end{equation*}
which is an $L_\infty$-morphism with respect to the product $L_\infty$-structure. 
Setting 
\begin{equation*}
  H(t) \colon
	\Dpoly(M)\oplus \image [D,D^{-1}] \ni 
	(D_1,D_2)
	\longmapsto 
	(0,-D^{-1}D_2) \in 
	\Dpoly(M)\oplus \image [D,D^{-1}]
\end{equation*}
we obtain again 
\begin{align*}
  \frac{\D}{\D t} M(t)
	& =
	\pr_{\image [D,D^{-1}]} 
	=
	0 \oplus (DD^{-1}+D^{-1}D)\\
	& =
	-\del \oplus(\del_M + D)\circ H(t) - 
	H(t) \circ (\del \oplus (\del_M + D))\\
	& =
	\widehat{Q}^1(H(t)\vee \exp(M(t))).
\end{align*} 
Therefore, it follows that 
\begin{equation*}
  L(t)
	=
	L^{-1} \circ M(t) \circ L 
	\colon 
	\Omega(M,\Dpolyres)
	\longrightarrow 
	\Omega(M,\Dpolyres)
\end{equation*}
encodes the homotopy between 
\begin{equation*}
  L(0)
	=
	\tau\circ \nu^{-1} \circ P
	\quad \text{ and } \quad 
	L(1)
	=
	\id.
\end{equation*}
But this implies with Proposition~\ref{prop:CompofHomotopicHomotopic}
\begin{equation*}
  P \circ e^{-[h,\argument]}
	\sim
	P \circ e^{-[h,\argument]} \circ \tau' \circ (\nu')^{-1} \circ P'
	=
	P'
\end{equation*}
and the statement is shown.
\end{proof}

Combining all the above statements we can show that the globalizations 
with respect to different covariant derivatives are homotopy equivalent. 

\begin{theorem}
\label{thm:GlobalwithModifareHomotopic}
  Let $(\nabla,r)$ and $(\nabla',r')$ be two pairs of globalization data.
  Then the formalities constructed via Dolgushev's globalization 
	and the globalized formalities via the $L_\infty$-projection are 
	all are homotopic, i.e. one has
	\begin{equation}
	  \boldsymbol{U}^{(\nabla,r)}
		\sim 
		F^{(\nabla,r)}
		\sim
		F^{(\nabla',r')}
		\sim
		\boldsymbol{U}^{(\nabla',r')} .
	\end{equation}
\end{theorem}
\begin{proof}
By Proposition~\ref{prop:PropertiesofHomotopicMorphisms} we already know 
that compositions of homotopic $L_\infty$-morphisms with DGLA morphisms are 
homotopic, which yields
\begin{align*}
  U
	\sim
	\mathcal{U}^{B} \circ \tau \circ \nu^{-1}
	\sim
	e^{-[h,\argument]}\circ \mathcal{U}^{B'} \circ e^{[h,\argument]}
	\circ\tau \circ \nu^{-1} 
	=
	e^{-[h,\argument]}\circ \mathcal{U}^{B'} \circ \tau' \circ (\nu')^{-1} 
	\sim
	e^{-[h,\argument]}\circ U'.
\end{align*}
It follows with Lemma~\ref{lemma:ProjHomotopic}, Proposition~\ref{prop:CompofHomotopicHomotopic} and 
Proposition~\ref{prop:preCompofHomotopicHomotopic}
\begin{align*}
  \boldsymbol{U}^{(\nabla,r)}
	=
	\nu\circ \sigma \circ U
	& =
	P \circ U
	\sim
  P \circ \mathcal{U}^B \circ\tau \circ \nu^{-1}
	\sim
	P \circ e^{-[h,\argument]}\circ \mathcal{U}^{B'} \circ \tau' \circ (\nu')^{-1}\\
	&\sim 
	P' \circ \mathcal{U}^{B'} \circ \tau' \circ (\nu')^{-1}
	\sim 
	P' \circ U'
	=
	\nu'\circ \sigma \circ U'
	=
	\boldsymbol{U}^{(\nabla',r')}
\end{align*}
and the theorem is shown.
\end{proof}

\begin{corollary}
Let $M$ be a smooth manifold and let $(\nabla,r)$ 
be a globalization data. For every coordinate patch 
$(U,x)$ 
	\begin{align*}
	F^{(\nabla,r)}\at{U}
	\sim
	K\at{U},
	\end{align*}
	holds, 
where $K$ denotes the Kontsevich formality on 
$\mathbb{R}^d$, and 
where $d$ is the dimension of $M$. 
\end{corollary}
\begin{proof}
The formalities themselves are differential operators and
can therefore be restricted to open neighbourhoods. Moreover, 
the Kontsevich formality coincides with the Dolgushev formality on 
$\mathbb{R}^d$ for the choice of the canonical flat covariant derivative
and $r=0$. 
\end{proof}
	
This allows us to recover 
\cite[Theorem~2.6]{bursztyn.dolgushev.waldmann:2012a}, i.e. that 
the induced maps on equivalence classes of Maurer-Cartan elements 
are independent of the choice of the covariant derivative. It 
implies in particular that globalizations with respect to different covariant 
derivatives lead to equivalent star products.

\begin{corollary}
  \label{cor:StarProdIndependentOfConnection}
  The induced map $	\Def(\Tpoly(M)[[\hbar]]) 	\longrightarrow	
	\Def(\Dpoly(M)[[\hbar]])$  
	does not depend on the choice of a covariant derivative.
\end{corollary}
\begin{proof}
The statement follows directly from Theorem~\ref{thm:GlobalwithModifareHomotopic} 
and Proposition~\ref{prop:PropertiesofHomotopicMorphisms}.
\end{proof}

Finally, note that Theorem~\ref{thm:GlobalwithModifareHomotopic} also 
holds in the equivariant setting of an action of a 
Lie group $\group{G}$ on $M$ with $\group{G}$-invariant torsion-free 
covariant derivatives $\nabla$ and $\nabla'$.

\begin{proposition}
  Let $\group{G}$ act on $M$ and consider two pairs of globalization data 
	$(\nabla,r)$ and $(\nabla',r')$, where $\nabla$ and $\nabla'$ are two 
	$\group{G}$-invariant torsion-free covariant derivatives and where 
	$r$ and $r'$ are $\group{G}$-invariant. Then 
	the formalities are equivariant and equivariantly homotopic
	\begin{equation}
	  \label{eq:EquivariantFormalitiesHom}
	  \boldsymbol{U}^{(\nabla,r)}
		\sim_{\group{G}} 
		F^{(\nabla,r)}
		\sim_{\group{G}}
		F^{(\nabla',r')}
		\sim_{\group{G}}
		\boldsymbol{U}^{(\nabla',r')} ,
	\end{equation}
	i.e. all paths encoding the equivalence relation from 
	\eqref{eq:EquivMCElements} are $\group{G}$-equivariant.
\end{proposition}
\begin{proof}
The formalities are equivariant since all involved maps are 
\cite[Theorem~5]{dolgushev:2005a}. Moreover, 
$\mathcal{U}^B$ and $e^{-[h,\argument]}\circ 
\mathcal{U}^{B'}\circ e^{[h,\argument]}$ are equivariantly homotopic by 
the explicit form of the homotopy from 
Proposition~\ref{prop:TwistMorphHomEqu}. Moreover, again by 
\cite[Theorem~5]{dolgushev:2005a} we know that $U$ and 
$\mathcal{U}^B\circ \tau\circ \nu^{-1}$ are equivariantly homotopic, the 
same holds for the $(\nabla',r')$ case. Thus by 
Theorem~\ref{thm:GlobalwithModifareHomotopic} it only remains to show that 
$P\circ e^{-[h,\argument]}$ and $P'$ are equivariantly homotopic. But 
this follows directly from Lemma~\ref{lemma:ProjHomotopic} since all involved 
maps are equivariant. 
\end{proof}

In the case of proper Lie group actions one knows that invariant covariant 
derivatives always exist and one has even an invariant 
Hochschild-Kostant-Rosenberg theorem, compare 
\cite[Theorem~5.11]{miaskiwskyi:2018a:pre}. Thus the $L_\infty$-morphisms 
from \eqref{eq:EquivariantFormalitiesHom} restrict to the 
invariant DGLAs and one obtains homotopic formalities 
from $(\Tpoly(M))^\group{G}$ to $(\Dpoly(M))^\group{G}$.

\section{Final Remarks}

In \cite[Thm.~6]{dolgushev:2005b} it is proven that the construction 
of the formality $ \boldsymbol{U}^{(\nabla,0)}$ is functorial for 
diffeomorphisms of pairs $(M,\nabla)$. 
More explicitly, the objects of the source category 
are pairs $(M,\nabla)$ of manifolds with 
torsion-free covariant derivatives, and a morphism from 
$(M,\nabla)$ to $(M',\nabla')$ is a diffeomorphism
$\phi\colon M\to M'$ such that 
	\begin{align*}
	\phi_*(\nabla_X Y)=\nabla'_{\phi_* X}\phi_* Y
	\end{align*}
for all $X,Y\in \Secinfty(TM)$. The target category is the category of triples 
$(A,B,F)$, where $A,B$ are $L_\infty$-algebras and where $F\colon A\to B$ 
is an $L_\infty$-quasi-isomorphism. A morphism is a pair 
$(f,g)\colon (A,B,F)\to (A',B',F')$ consisting of two $L_\infty$-morphisms 
$f\colon A\to A'$ and $g\colon B\to B'$ such that 
	\begin{center}
	\begin{tikzcd}
	A\arrow[r, "F"]\arrow[d,swap, "f"] & B\arrow[d, "g"]\\
	A'\arrow[r, "F'"']  & B'
	\end{tikzcd}
	\end{center}
commutes. The functor from \cite{dolgushev:2005b} is hence given by 
	\begin{align*}
	(M,\nabla)\longrightarrow (T_\poly(M), D_\poly(M),\boldsymbol{U}^{(\nabla,0)}),
	\end{align*}	 
mapping diffeomorphisms to push-forwards of vectorfields resp. differential 
operators. Our investigations from above 
lead now to a functor from the category of manifolds with diffeomorphisms into 
a category as above but with morphisms being homotopy classes of 
$L_\infty$-quasi-isomorphims and $L_\infty$-morphisms, respectively. It is
 given by 
	\begin{align*}
	M\longrightarrow ((T_\poly(M), D_\poly(M),[\boldsymbol{U}^{(\nabla,0)}]),
	\end{align*}	   
where $[\argument]$ indicates the homotopy class and where $\nabla$ is an 
arbitrary torsion-free connection. 

Moreover, for a Lie group $G$, we can consider the source category of 
$G$-manifolds with proper $G$-actions and with 
equivariant diffeomorphisms to get the functor
	\begin{align*}
	M\to (T_\poly(M)^G,D_\poly(M)^G, [\boldsymbol{U}^{(\nabla,0)}]).
	\end{align*}	
Here $\nabla$ is an arbitrary $G$-invariant torsion-free connection.

%
%
\bibliographystyle{nchairx}

\end{document}